\documentclass[]{elsarticle}

\usepackage{lineno,hyperref}

\usepackage{graphicx}
\usepackage{bm}
\usepackage{multirow}
\usepackage{amsmath}
\usepackage{amsfonts}
\usepackage{amssymb}
\usepackage{amsthm}
\usepackage{secdot}
\usepackage{subcaption}
\usepackage{lscape}
\usepackage{xcolor}
\usepackage{enumitem}

\journal{Comput.~Methods Appl.~Mech.~Engrg.}

\theoremstyle{plain}
\newtheorem{theorem}{Theorem}[section]
\newtheorem{lemma}[theorem]{Lemma}

\theoremstyle{definition}
\newtheorem{definition}[theorem]{Definition}

\theoremstyle{remark}
\newtheorem{remark}{Remark}

\theoremstyle{assumption}
\newtheorem{assumption}[theorem]{Assumption}

\numberwithin{equation}{section}
\numberwithin{theorem}{section}
\numberwithin{remark}{section}

\newcommand{\llbracket}{\left[\!\left[}
\newcommand{\rrbracket}{\right] \! \right]}

\begin{document}

\begin{frontmatter}

\title{Spline-Based Solution Transfer for Space-Time Methods in 2D+$t$}

\author{Logan Larose\corref{mycorrespondingauthor}}

\cortext[mycorrespondingauthor]{Corresponding author}
\ead{lfl5340@psu.edu}

\author{Jude T. Anderson }
\author{David M. Williams} 

\address{Department of Mechanical Engineering, The Pennsylvania State University, University Park, Pennsylvania 16802}

\address{Laboratories for Computational Physics and Fluid Dynamics, Naval Research Laboratory, Washington, DC 20375}

\fntext[fn1]{Distribution Statement A: Approved for public release. Distribution is unlimited.}

\begin{abstract}
This work introduces a new solution-transfer process for slab-based space-time finite element methods. The new transfer process is based on Hsieh-Clough-Tocher (HCT) splines and satisfies the following requirements: (i) it maintains high-order accuracy up to 4th order, (ii) it preserves a discrete maximum principle, (iii) it asymptotically enforces mass conservation, and (iv) it constructs a smooth, continuous surrogate solution in between space-time slabs. While many existing transfer methods meet the first three requirements, the fourth requirement is crucial for enabling visualization and boundary condition enforcement for space-time applications. In this paper, we derive an error bound for our HCT spline-based transfer process. Additionally, we conduct numerical experiments quantifying the conservative nature and order of accuracy of the transfer process. Lastly, we present a qualitative evaluation of the visualization properties of the smooth surrogate solution. 
\end{abstract}

\begin{keyword}
Solution transfer \sep Space-time \sep Finite element methods \sep Hsieh-Clough-Tocher splines \sep High order \sep Visualization
\end{keyword}

\end{frontmatter}


\section{Introduction}
\label{sec;introduction}

Solution transfer is a well-established aspect of scientific computing, see~\cite{farrell2009conservative,farrell2011conservative} and the references therein. The demand for high-fidelity solution transfer is increasingly pressing as automatic anisotropic mesh refinement continues to facilitate greater computational efficiency~\cite{marcum2014aligned} and produce more accurate results than post-processing by enrichment or moving points~\cite{george1991creation}. Solution transfer allows a simulation to be resumed on a new, adapted mesh---called the \textit{target mesh}---by transferring data from its most recent state on a pre-existing mesh---called the \textit{source mesh}. For time-dependent problems, mitigating the accumulation of error introduced by transfer/interpolation is essential for maintaining conservation. This is true for the classical method-of-lines approaches \emph{and} the more recent slab-based, space-time approaches. The latter space-time approaches are the main motivation for this paper, although, the work presented here is not limited to this context. In what follows, we explain our motivation in more detail.

\subsection{Motivation}

Space-time methods are inherently expensive, as they require an extra dimension, relative to traditional method-of-lines approaches. In order to address this issue, researchers such as Hulbert, Hughes, Mallet, Shakib and coworkers~\cite{hulbert1990space,hughes1987new,shakib1991new,hughes1996space,shakib1991new2} divided the space-time domain into space-time slabs. As a result, space-time problems can be solved on smaller subdomains (time slabs) in a sequential fashion. This approach to space-time problems is fairly mature, and the interested reader is invited to consult~\cite{anderson2023surface} for an overview of recent work in this area. The standard slab-based space-time approach does not require a sophisticated solution-transfer strategy, as the meshes on adjacent space-time slabs are conforming at the interface between the slabs. In this case, the solution-transfer process is trivial. However, due to recent advances in space-time mesh adaptation~\cite{caplan2019four,caplan2020four,nishikawa2020adaptive,caplan2022parallel,padway2022adaptive,belda2023conformal,gehring2023constant,diening2023adaptive,nishikawa2023time}, there is strong interest in creating solution-adapted meshes on space-time slabs. In this case, the meshes on consecutive space-time slabs are not guaranteed to be conforming. For these adapted, slab-based, space-time methods, at least two different types of solution transfer are required:
\begin{itemize}
    \item Solution transfer between adapted space-time meshes on a single space-time slab. These transfers are between pairs of $n$D+$t$ meshes.  
    \item Solution transfer between the terminating surface of a space-time mesh on a given slab, and the initial surface of a space-time mesh on the next slab. These transfers are between pairs of $n$D meshes.  
\end{itemize}
In this article, we are primarily interested in solution transfer of the second type, for cases in which $n=2$. It may seem natural to apply an existing solution-transfer approach to our space-time problems. However, solution-transfer methods for space-time slabs have unusual requirements: namely, a) they should support \emph{solution visualization}, and b) they should enable the \emph{enforcement of space-time boundary conditions}. 

In regards to visualization, one is often interested in viewing the solution on the initial or terminating plane of a space-time slab. However, if the solution data is discontinuous, this solution and its derivatives may not be smooth, and may be difficult to interpret. Therefore, we desire a strategy for creating a smooth solution (and derivatives) at the interfaces between space-time slabs. This smooth solution can be constructed as a part of the solution-transfer process. We note that there are more complicated approaches for visualizing the solution at an arbitrary point in time. In this case, one must extract data from the simulation using an \emph{intersection} approach. This strategy has been attempted by multiple authors, including Caplan~\cite{caplan2019four} and Belda Ferr{\'\i}n et al.~\cite{belda2020visualization}. However, this level of flexibility for visualization is beyond the scope of the current work.  

In regards to boundary conditions, one must enforce a suitable condition on the solution at the initial plane (hyperplane) of each space-time slab. This space-time boundary condition is taken from the solution on the terminating surface of the previous space-time slab. The process of enforcing this boundary condition can be complex, especially if the solution data on the terminating surface of the previous space-time slab is discontinuous. This is not immediately obvious, but can be illustrated with a simple example. Suppose that the surface meshes at the interfaces between adjacent time slabs are not conforming; in this case, the discontinuities which appear along the edges or faces of the elements in one surface mesh will overlap with the interiors of elements on the adjacent surface mesh. This may negatively impact the anisotropic adaptive-meshing process on the next space-time slab. In particular, there is a strong possibility that the meshing process will identify the discontinuities in the initial data, and then adapt the mesh to these (likely spurious) features. Therefore, smooth initial data is desirable for boundary condition enforcement in adaptive, slab-based, space-time methods.  

The issue described above, is already known to arise for three-dimensional \emph{steady} problems with spatial boundaries. In this context, the boundaries of the domain are composed of CAD surfaces, often called BREPs (boundary representations). Unfortunately, BREPs frequently contain degeneracies and singularities, as well as artifacts which arise from trimming~\cite{gammon2018review}. An anisotropic adaptive-meshing process will often identify these unimportant geometric features, and then attempt to resolve them~\cite{park2019geometry,park2021boundary}. In order to address this issue, the creators of the mesh adaptation software EPIC (Edge Primitive Insertion and Collapse~\cite{michal2012anisotropic}) developed a quadratic geometry surrogate based on the approach of~\cite{nagata2005simple}. In addition, the creators of  FEFLO.A~\cite{loseille2013cavity,loseille2017unique} created a cubic geometry surrogate~\cite{loseille2023p3}. Both approaches replace the BREP definitions with a smoother, surrogate geometry for spatial boundary condition enforcement. Inspired by this work, it is natural to construct a smooth, surrogate solution for enforcing space-time boundary conditions. 

In summary, solution-transfer methods for space-time applications have additional requirements, which are generally not satisfied by existing solution-transfer methods. In particular, we contend that the space-time solution-transfer process between slabs should create a smooth surrogate solution which facilitates boundary condition enforcement and visualization.

For the sake of completeness, in what follows, we will briefly review state-of-the-art techniques for solution transfer, despite their lack of immediate applicability to space-time methods. 

\subsection{Literature Review}

Linear interpolation is perhaps the most well-known solution-transfer approach. Unfortunately, standard linear interpolation procedures fail to preserve the conservative properties of a solution and introduce significant error~\cite{alauzet20073d,alauzet2010p1}. The shortcomings of linear interpolation have led to several developments in conservative interpolation in order to meet the rising demand for accurate and conservative transfer.

The technique of $L_2$-\emph{projection} or \emph{Galerkin projection} is a natural way to address the issues of conservation and accuracy which arise for linear interpolation. Some of the earliest work on $L_2$-projection appears to be that of George and Borouchaki~\cite{georgedelaunay}, Geuzaine et al.~\cite{geuzaine1999galerkin}, and Parent et al.~\cite{parent2008using}. For this approach, one requires that the integration of the solution (and its moments) over the source mesh is equivalent to the integration of the transferred solution (and its moments) over the target mesh. 
This procedure is straightforward, and always yields the desired results when the numerical integration, i.e.~the quadrature, is exact. Of course, if the meshes are not perfectly aligned, the quadrature will not be exact, as quadrature rules are only exact for polynomials. Unfortunately, the functions being integrated are usually piecewise polynomials with discontinuities, and thus conservation errors are introduced. The most popular way to address this issue is via mesh intersections: namely, one computes the intersection of the source and target meshes, and thereafter, creates meshes on the intersected regions. Exact integration is then performed over the resulting submeshes. It appears that George and Borouchaki first proposed this idea~\cite{georgedelaunay}. However, the details of their algorithm have not been formally published. Fortunately, there are several popular variants of this approach, which we will review below.

Farrell et al.~\cite{farrell2009conservative,farrell2011conservative} introduces the concept of a `supermesh' which they perform their integration over, as opposed to the standard approach of integrating over the target mesh, (see for example, Geuzaine et al.~\cite{geuzaine1999galerkin}). Farrell et al.~define a supermesh as follows: given two arbitrary unstructured volume meshes $\mathcal{T}_{s}$ and $\mathcal{T}_{t}$ (source and target meshes) of the domain $\Omega$, a supermesh is a mesh over $\Omega$ which, (1) contains all of the nodes from the parent meshes ($\mathcal{T}_{s}, \mathcal{T}_{t}$), and (2) the intersection of an element of the supermesh with any element of the parent meshes is either an empty set (no overlap) or is the whole element (completely overlaps). The transfer method proposed by Farrell et al.~first computes the supermesh of the target and source meshes. Additionally, a mapping of elements from the supermesh to the source mesh, and a mapping of elements from the target mesh to the source mesh is produced. Next, for each element on the target mesh, the integral over the element is taken as the sum of the integral values of the elements on the supermesh which it contains. Finally, the transferred solution is computed via $L_2$-projection. During the final step of the method, the mass matrix of the linear system introduced by the projection is replaced by a lumped mass matrix. The replacement does not affect conservation, but does bound the resulting solution and adds artificial diffusion. Additional care is required to ensure that the impact of the diffusivity on the solution is negligible. For further information on this strategy, we refer the interested reader to~\cite{farrell2009conservative,farrell2011conservative} and the references therein.

Alauzet and Mehrenberger~\cite{alauzet2010p1,alauzet2016parallel} developed a `matrix-free' interpolation scheme which avoids the construction of a global supermesh or linear systems. Instead, their method iterates over each element in the target mesh one-at-a-time, and computes its intersections with elements in the source mesh. This process returns a set of polygonal or polyhedral subdomains. Each of these subdomains is then meshed with triangular or tetrahedral elements, respectively. By construction, each triangular or tetrahedral element of the subdomains is guaranteed to be completely contained within an element on the source mesh. Exact quadrature is employed in order to construct a Taylor-series polynomial on each element of the target mesh. In a natural fashion, this polynomial can be evaluated at any nodes (or vertices) which belong to the element. This method is guaranteed to maintain second order accuracy (i.e.~$\mathcal{P}_1$ exactness). In addition, a discrete maximum principle can be enforced by a limiting procedure which corrects the reconstructed solution. 

In addition, we note that Alauzet and Mehrenberger~\cite{alauzet2010p1,alauzet2016parallel} identified a convenient list of three requirements for the solution-transfer process:
\begin{enumerate}[label=(\roman*)]
    \item Second-order accuracy, i.e.~$\mathcal{P}_1$ exactness.
    \item Discrete maximum principle.
    \item Mass conservation.
\end{enumerate}
Their proposed method satisfies all three of these criteria, and performs excellently for its intended applications. However, for our space-time applications, we have created an expanded list of requirements:
\begin{enumerate}[label=(\roman*)]
    \item High-order accuracy, i.e.~$\mathcal{P}_k$ exactness for $k \leq 3$.
    \item Discrete maximum principle.
    \item Mass conservation.
    \item Smooth, continuous surrogate solution.
\end{enumerate}
The smooth, continuous surrogate solution is required for the enforcement of boundary conditions, and provides a convenient means of visualization, as we discussed previously. 




\subsection{An Alternative Approach}

In this work, we introduce a new solution-transfer method based on Hsieh-Clough-Tocher (HCT) splines. Broadly speaking, our method takes the original solution on the source mesh and constructs a smooth, surrogate solution by using averaging procedures and HCT interpolation techniques. Next, the smooth surrogate solution is transferred to the target mesh using a $L_2$-projection procedure. In order to ensure that the transferred solution remains well-behaved, we (optionally) enforce the discrete maximum principle on the transferred solution using the limiting procedure of Alauzet and Mehrenberger~\cite{alauzet2010p1,alauzet2016parallel}. Our approach is summarized below:

\begin{enumerate}
    \item Synchronize---i.e. carefully average---the solution and its first derivatives on the source mesh.
    \item Use HCT elements to interpolate the synchronized solution on the source mesh.
    \item Use $L_2$-projection to transfer the interpolated solution from the source mesh to the target mesh.
\end{enumerate}
In the process above, limiting can be enforced in step 3 as necessary. 

It is important to note that step 3 of our proposed method does not use the same $L_2$-projection strategy as Farrell et al.~\cite{farrell2009conservative,farrell2011conservative} or Alauzet and Mehrenberger~\cite{alauzet2010p1,alauzet2016parallel}. Instead of using mesh intersections to address conservation concerns, we use an element-subdivision strategy in which an element on the target mesh is subdivided into smaller elements, and quadrature rules are applied to these subelements. Our approach is the first step towards replacing the standard mesh intersection approach with an \emph{adaptive-quadrature} approach. From our perspective, adaptive quadrature is a more promising approach for addressing the conservation problem associated with solution transfer, especially as we consider the complexity of computing mesh intersections and meshing subdomains in higher-dimensional space, (i.e.~4D for most space-time applications). 


\subsection{Overview of the Paper}

This article begins by providing a detailed description of our HCT-based solution-transfer method in Section~\ref{sec;Implementation}. Next, Section~\ref{sec;theory} presents theoretical results which characterize the accuracy of the solution-transfer method. Thereafter, Section~\ref{sec;Results} presents the results of numerical experiments which assess the mass conservation and order of accuracy of the method. In addition, we demonstrate the effectiveness of our HCT-based method for visualization purposes. Section~\ref{sec;Conclusion} presents our conclusions and summarizes the findings of this work. Lastly, supplemental tables can be found in the Appendix.

\section{Implementation Details}
\label{sec;Implementation}

\subsection{Preliminaries}

The purpose of our solution-transfer method is to pass a solution from the terminating surface of a space-time slab to the initial surface of the next space-time slab. In order to fix ideas, let us introduce some notation. We refer to the original space-time slab as $Q_n$. This slab spans the interval from $[t_{n-1},t_{n}]$, and has a terminating surface mesh located at time $t_n^{-}$. This terminating surface mesh is the source mesh, and we denote it by mesh $a$ or $\mathcal{T}_a$. The next space-time slab is denoted by $Q_{n+1}$ and spans the interval from $[t_{n},t_{n+1}]$. It has an initial surface mesh located at time $t_{n}^{+}$. This initial surface mesh is the target mesh, and we denote it by mesh $b$ or $\mathcal{T}_b$. With this notation in mind, our objective is to transfer the solution $v_{h_a}$ from $\mathcal{T}_a$ located at $t_{n}^{-}$ to $\mathcal{T}_b$ located at $t_{n}^{+}$.

As we discussed earlier, the transfer process starts with an averaging procedure, in conjunction with HCT interpolation. This operation is denoted by $E(\cdot )$. The averaging and interpolation process forms a smooth, surrogate solution at time level $t_{n}$. This solution is denoted by $E(v_{h_a})$. Next, we perform an $L_2$-projection from this surrogate solution on to mesh $b$. This operation is denoted by $I_{h_b}(\cdot )$. The entire solution-transfer process, starting with $v_{h_a}$ and ending with $I_{h_b}(E(v_{h_a}))$ is illustrated in Figure~\ref{fig:slab_trans}.
\begin{figure}[h!]
    \centering
    \includegraphics[width=0.6\linewidth]{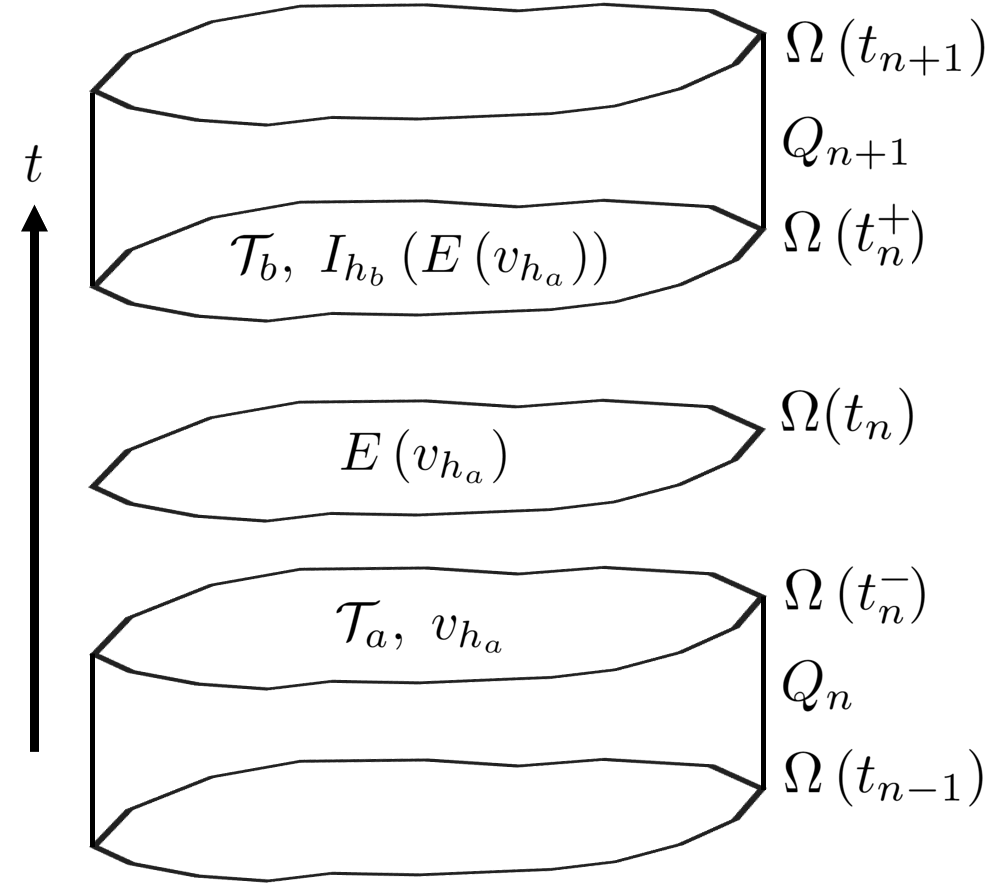}
    \caption{Transferring the solution $v_{h_a}$ on the terminating surface mesh $\mathcal{T}_a$ of the space-time slab $Q_n$ to the initial surface mesh $\mathcal{T}_b$ of the subsequent space-time slab $Q_{n+1}$. The smoothed solution is obtained using the smoothing operator $E(\cdot)$ on the solution $v_{h_a}$. This solution $E(v_{h_a})$ provides a surrogate solution at $t_n$ that is suitable for visualization and boundary condition enforcement. Projection operator $I_{h_b}(\cdot)$ transfers the smoothed solution to the initial grid of $Q_{n+1}$.}
    \label{fig:slab_trans}
\end{figure}

These steps will be discussed in greater detail in what follows. Additionally, throughout our discussion of the implementation, we assume that the solution degrees of freedom are laid out as shown in Figure~\ref{fig:P1_P2_dof}, for polynomial degrees $k = 1, 2,$ and $3$.
\begin{figure}[h!]
    \centering
    \includegraphics[width = 0.7\textwidth]{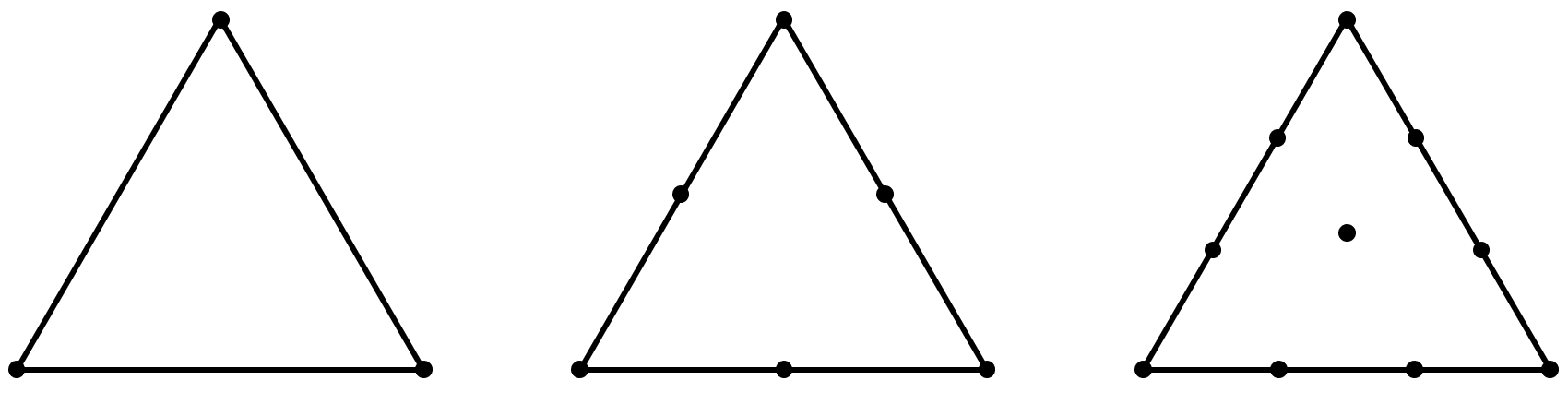}
    \caption{Left, the degrees of freedom for a triangular element with $k = 1$. Middle, the degrees of freedom for a triangular element with $k = 2$. Right, the degrees of freedom for a triangular element with $k = 3$.}
    \label{fig:P1_P2_dof}
\end{figure}

\subsection{Synchronization}
\noindent The first step towards transferring the solution from the source grid to the target grid requires preparing the solution for HCT interpolation. This is because the effectiveness of HCT interpolation is maximized if it is supplied with smooth input data. Ideally, this data would be $\mathcal{C}^1$-continuous throughout the domain in order to maintain inter-element continuity of the solution and its first derivatives. This is the primary goal of our synchronization procedure.

Given a discontinuous solution, a smoothing operator can be used to synchronize (average) the solution and its gradient to the nodes; in addition the gradient of the solution can be averaged at the edge midpoints. 
For the $k = 1$ case, the solution is stored only at the vertices. To obtain the derivatives at each midpoint, we can compute the derivatives at the vertices within each element, and copy the derivative values to the midpoints. This operation is valid because the derivatives are constant over each element when $k=1$. The process is illustrated in red, in Figure~\ref{fig:synch}. 
For the $k = 2$ and $k =3$ cases, the polynomial solution on each element must be differentiated and directly evaluated at the vertices and edge midpoints.

Given derivative values at the vertices and midpoints, and solution values at the vertices, we average the solution and derivative information between elements in order to create 
$\mathcal{C}^{1}$-continuous data throughout the domain. 
This process is shown in blue, in Figure \ref{fig:synch}.
\begin{figure}[h!]
    \centering
    \includegraphics[width = 0.6\textwidth]{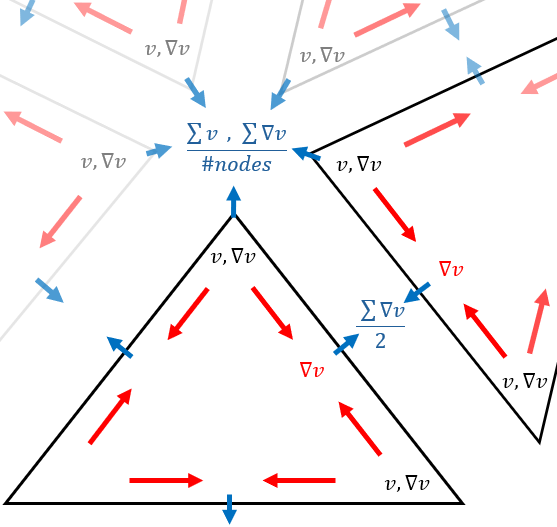}
    \caption{The synchronization procedure for a discontinuous finite element solution, for the $k=1$ case.}
    \label{fig:synch}
\end{figure}

\subsection{HCT interpolation}
\noindent Our method uses the HCT-C element, (complete Hsieh-Clough-Tocher element), in accordance with the terminology of Bernadou and Hassan~\cite{bernadou1981basis}. This element is illustrated in Figure \ref{fig:HCT_tri}.
\begin{figure}[h!]
    \centering
    \includegraphics[width = 1.0\textwidth]{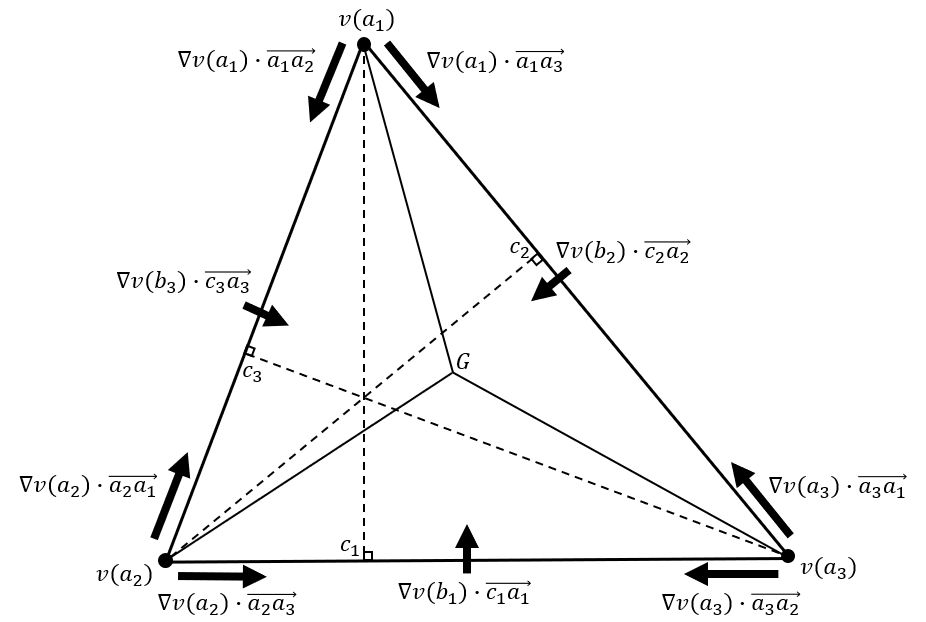}
    \caption{The HCT-C triangle as defined by Bernadou and Hassan~\cite{bernadou1981basis}, shown with its degrees of freedom, $\Sigma_T$.}
    \label{fig:HCT_tri}
\end{figure}
A single triangle element denoted by $T$, is defined by vertices $\{ a_1,a_2,a_3 \}$. Note that henceforth, all indices will assume values of 1, 2, 3, modulo 3, (except anything that returns 0 is mapped to 3). For example, index 4 maps to index 1, index 5 to index 2, and index 6 to index 3. The subtriangles $T_i$ are formed by $\{G, a_{i+1}, a_{i+2} \}$, where $G$ is the barycenter of $T$. The edges of $T$ are formed such that, $\ell_i := \{a_{i+1}, a_{i+2}\}$. The midpoint of edge $\ell_i$ is denoted by $b_i$, and the orthogonal projection of vertex $a_i$ onto $\ell_i$ is denoted by $c_i$. 

Consider $\mathcal{HCT}(T)$, the set of piecewise cubic polynomials defined on the element $T$ such that
\begin{align*}
    \mathcal{HCT}(T) = \{p \in \mathcal{C}^1(T): p\rvert_{T_i} \in \mathcal{P}_3(T_i), i = 1,2,3 \}.
\end{align*}
%
Recall that after synchronization,  $v\in\mathcal{C}^{1}(\Omega)$.
We can project this function onto the HCT space, $\mathcal{HCT}(T)$, in a piecewise fashion, by defining degrees of freedom for $v$, denoted by $\Sigma_T$. The set $\Sigma_T$ consists of the function $v$ evaluated at the vertices, $v(a_i)$; the directional derivative of the function evaluated at the vertices in the direction of each edge emanating from that vertex, $\nabla v(a_i)\cdot \overrightarrow{a_i a_{i+2}}$, and $\nabla v(a_i)\cdot \overrightarrow{a_i a_{i+1}}$; lastly, the directional derivative of the function evaluated at the midpoint of each edge along the direction of the orthogonal projection, $\nabla v(b_i)\cdot \overrightarrow{c_ia_i}$. In summary,
\begin{align*}
    \Sigma_T = \{ v(a_i), \nabla v(a_i)\cdot \overrightarrow{a_i a_{i+2}}, \nabla v(a_i)\cdot \overrightarrow{a_i a_{i+1}}, \nabla v(b_i)\cdot \overrightarrow{c_ia_i}, i = 1,2,3 \}.
\end{align*}
Bernadou and Hassan~\cite{bernadou1981basis} follow Argyris et al.~\cite{argyris1968tuba} and introduce `eccentricity parameters', $E_i$, for each triangle $T$. Using only the Cartesian coordinates of the vertices of $T$, these parameters define the slope normal to the midpoint nodes and are defined as follows
\begin{align}
    E_i = \frac{\|\ell_{i+2}\|^2-\|\ell_{i+1}\|^2}{\|\ell_i\|^2},
    \label{eqn: eccparam}
\end{align}
where $\|\ell_i\|$ is the Euclidean norm of the edge $\ell_i$. The eccentricity parameters, along with barycentric or `area' coordinates are essential for calculating the basis functions for $T_i$.  With the preliminaries established, the basis functions and interpolant on $T_i$ are given by Bernadou and Hassan~\cite{bernadou1981basis} as follows
\begin{align}
    \pi_{T_i}v = \sum_{j=i}^{i+2}\left[v(a_i)r_{i,j}^0 + \nabla v(a_i)\cdot \overrightarrow{a_i a_{i+2}}r_{i,j,j+2}^1 + \nabla v(a_i)\cdot \overrightarrow{a_i a_{i+1}}r_{i,j,j+1}^1+\nabla v(b_i)\cdot \overrightarrow{c_ia_i}r_{\perp i, j}^1\right],
\end{align}
where $\pi_{T_i}v \in \mathcal{P}_3(T_i)$, and $r_{i,j}^0, r_{i,j,j+2}^1, r_{i,j,j+1}^1$,  etc.~are basis functions for the subtriangle $T_i$. The precise definitions of these basis functions can be found in~\cite{bernadou1981basis,bernadou1980triangles}.

\subsection{$L_2$-Projection}
\noindent In order to facilitate the $L_2$-projection, quadrature points are generated on each individual element of the target mesh and passed into the HCT algorithm. Perhaps the most important step of the transfer process is identifying the elements of the source grid which contain the quadrature points of elements of the target grid. For this operation, there are several strategies. One is recommended by Alauzet and Mehrenberger~\cite{alauzet2010p1}, another which was used in this study is a bounding volume hierarchy or BVH~\cite{stich2009spatial}, which organizes the grid into a tree structure with the most primitive data type---nodes---as the leaves. It is well known that the BVH method is not good at identifying points that lie near an element boundary. When this problem arises, the BVH method will indicate that a particular quadrature point (or group of points) does not belong to any of the elements on the source mesh. In order to rectify this issue, we use a robust matrix determinant method for determining where the point (or points) are located. 

The method works as follows: 
Assume triangle $T_m$ is defined by points, $m_1 = \{x_1,y_1\}$, $m_2 = \{x_2,y_2\}$, $m_3 = \{x_3,y_3\}$, and point $o =\{x_4,y_4\}$ is a point which could be inside or outside of $T_m$. To determine whether $o$ is inside of $T_m$, the area of $T_m$ is compared with the three subtriangles formed by connecting the edges of $T_m$ with $o$, as shown in Figure~\ref{fig:area_check}.
\begin{figure}[h!]
\centering
    \begin{subfigure}{0.5\textwidth}
        \centering
        \includegraphics[width=1.0\textwidth]{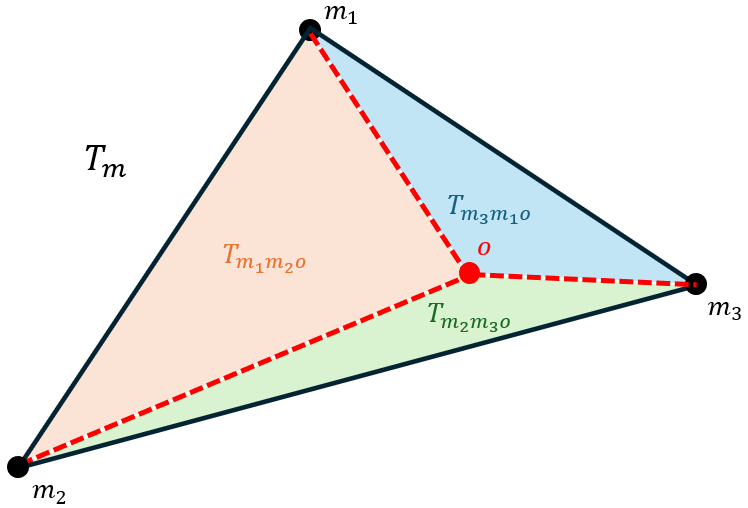}
    \end{subfigure}%
    \begin{subfigure}{0.5\textwidth}
        \centering
        \includegraphics[width=1.0\textwidth]{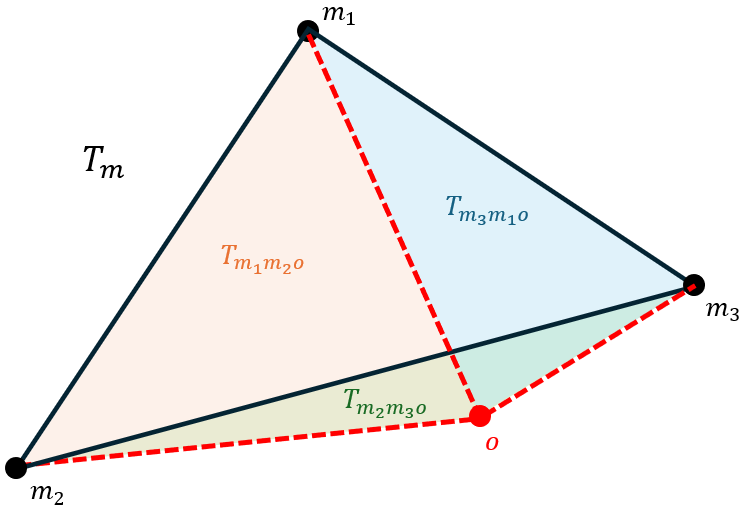}
    \end{subfigure}
    \caption{Left, triangle $T_m$ and three subtriangles formed by $o$, where $o$ is inside $T_m$. Right, triangle $T_m$ and three subtriangles formed by $o$, where $o$ is outside of $T_m$.}
    \label{fig:area_check}
\end{figure}
If the point $o$ is located inside of $T_m$, then 
\begin{align*}
    |T_m| = | T_{m_1m_2o}|+|T_{m_2m_3o}|+|T_{m_3m_1o}|,
\end{align*}
or equivalently, the areas of the subtriangles add up to the area of $T_m$. Otherwise, if $o$ is outside of $T_m$, then 
\begin{align*}
    |T_m| < |T_{m_{1}m_{2}o}|+|T_{m_{2}m_{3}o}|+|T_{m_{3}m_{1}o}|,
\end{align*}
which is illustrated by the right part of Figure~\ref{fig:area_check}. The matrix determinant is used to calculate the areas of triangle $T_m$ and its subtriangles in a robust manner. 

Once a  quadrature point is located within an element on the source mesh (see Figure \ref{fig:find_quad}), we must determine the HCT subelement, $T_i$, to which it belongs. This can be ascertained using the matrix determinant method mentioned previously. 
\begin{figure}[h!]
\centering
    \begin{subfigure}{0.5\textwidth}
        \centering
        \includegraphics[width=1.0\textwidth]{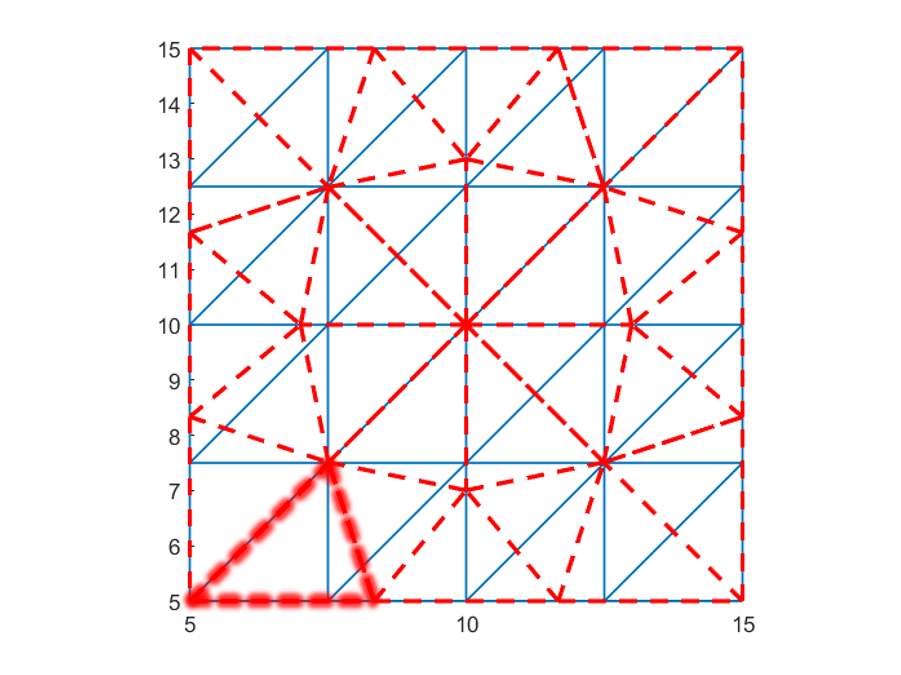}
    \end{subfigure}%
    \begin{subfigure}{0.5\textwidth}
        \centering
        \includegraphics[width=1.0\textwidth]{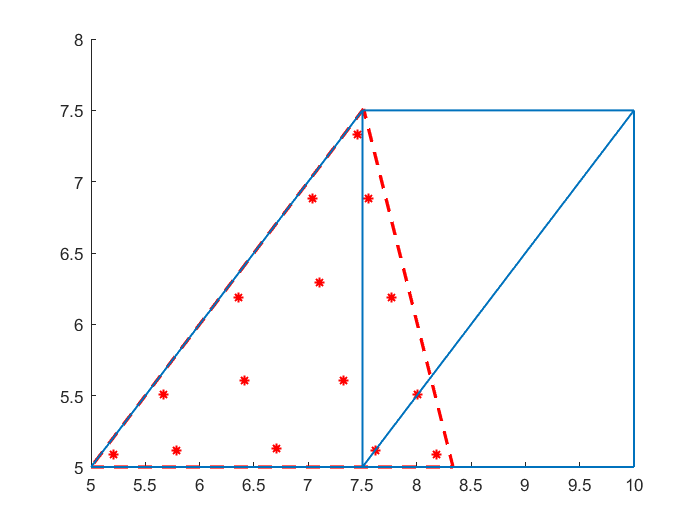}
    \end{subfigure}
    \caption{Left, the source grid (blue), with the corresponding target grid displayed on top (red). Right, the quadrature points which belong to the target element highlighted in dark red in the left image, and the corresponding source elements which also contain the quadrature points.}
    \label{fig:find_quad}
\end{figure}
Next, the HCT solution is evaluated at the quadrature points. In turn, the quadrature points are used to transfer the surrogate solution to the nodes of the target mesh via $L_2$-projection. 

The $L_2$-projection procedure is as follows. First, we construct a linear system, $M \hat{f} = B$, where $M$ is the `mass matrix',  $\hat{f}$ is a vector of coefficients for the modal basis functions $\psi_j$, and $B$ is the `forcing vector'. The mass matrix is defined such that
\begin{align}
    M_{ij} = \int_{T} \psi_i\left(\vec{r}(\vec{x})\right)\psi_j\left(\vec{r}(\vec{x})\right)\hspace{1mm}d\vec{x} =  \frac{A_T}{A_R}\int_{T_R} \psi_i\left(\vec{r}\right)\psi_j\left(\vec{r}\right)\hspace{1mm}d\vec{r},
    \label{lhs_integral}
\end{align}
where $\vec{x}$ are the physical space coordinates,
\begin{align}
    \vec{r} = \lambda_1\vec{\omega}_1+\lambda_2\vec{\omega}_2+\lambda_3\vec{\omega}_3,
\end{align}
are the reference space coordinates, and where $A_T/A_R$ is the ratio of the area of the physical triangle, $T$, to the area of the reference triangle, $T_R$. Here, $\vec{\omega}_{\ell}$ for $\ell =1, 2, 3$ is the vector of coordinates for each vertex of the reference space triangle, and $\{\lambda_1, \lambda_2, \lambda_3\}$ is the set of barycentric coordinates for a generic point in reference space.
The forcing vector on the right hand side is defined such that
\begin{align}
    B_i = \frac{A_T}{A_R} \int_{T_R} v\left(\vec{x}(\vec{r})\right)\psi_i \left(\vec{r}\right) \hspace{1mm} d\vec{r}.
    \label{rhs_integral}
\end{align}
%
%
A unique set of coefficients for the modal basis functions are obtained on each element by solving
\begin{align}
    \hat{f} = M_{ij}^{-1}B.
\end{align}
Note, the integrals in Eqs.~\eqref{lhs_integral} and \eqref{rhs_integral} are not computed analytically, rather a quadrature rule of sufficient strength is used to compute these integrals in practice. Once the coefficients of the modal basis functions are computed, the solution can be transferred to the nodes of the target mesh by taking the sum of the basis functions evaluated at the nodes in reference space times the computed coefficients.



The accuracy and conservation properties of the $L_2$-projection method are heavily dependent on the quadrature rule (or rules) that are used. This opens up a parallel discussion regarding how certain parameters of finite element methods can impact the accuracy of the solution. For finite element methods, one may enhance the accuracy through two types of refinement: (i) $p$-refinement, i.e., increasing the order of the polynomial representation---for example, upgrading from $k = 1$ elements to $k = 2$ elements, and (ii) $h$-refinement, i.e., size-based refinement of the elements---for example, splitting an element into four subelements to increase resolution. These dials also exist for quadrature rules.
\begin{figure}[h!]
\centering
    \begin{subfigure}{0.33\textwidth}
        \centering
        \includegraphics[width=1.0\textwidth]{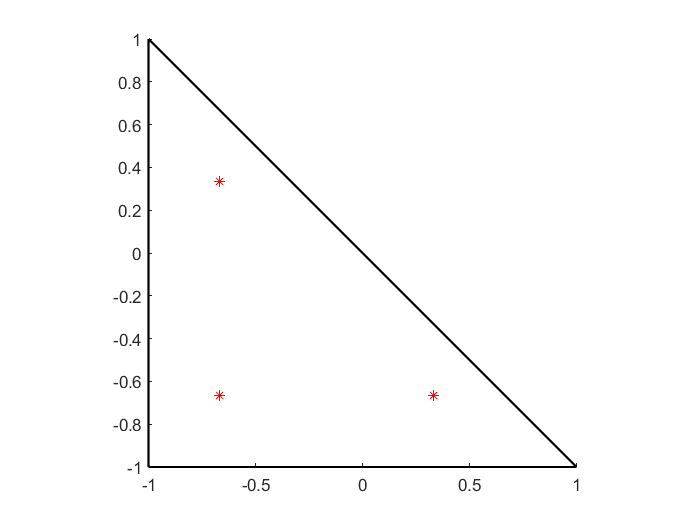}
    \end{subfigure}%
    \begin{subfigure}{0.33\textwidth}
        \centering
        \includegraphics[width=1.0\textwidth]{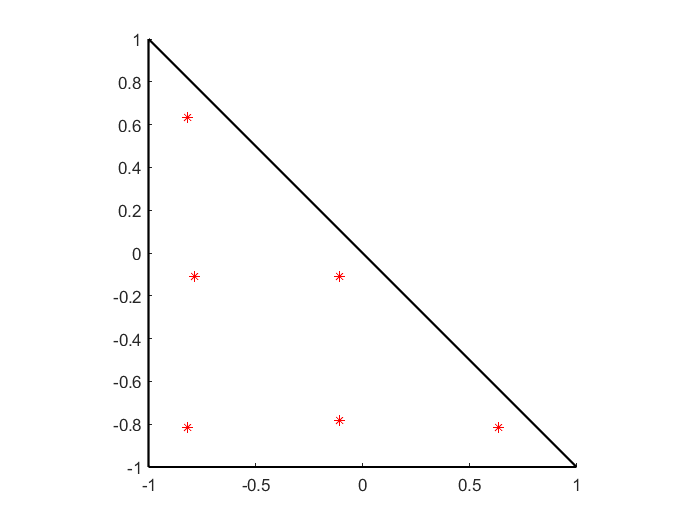}
    \end{subfigure}
    \begin{subfigure}{0.33\textwidth}
        \centering
        \includegraphics[width=1.0\textwidth]{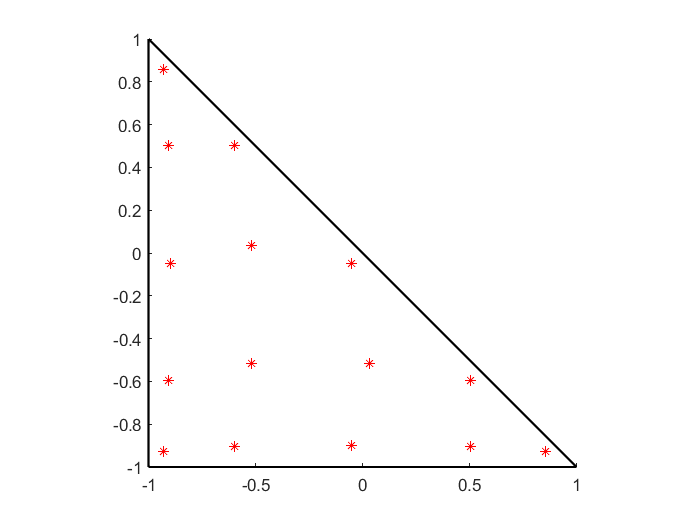}
    \end{subfigure}
    \begin{subfigure}{0.33\textwidth}
        \centering
        \includegraphics[width=1.0\textwidth]{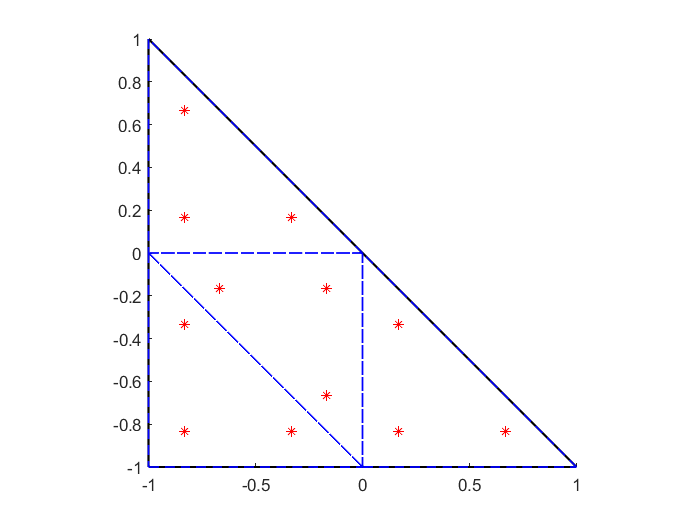}
    \end{subfigure}%
    \begin{subfigure}{0.33\textwidth}
        \centering
        \includegraphics[width=1.0\textwidth]{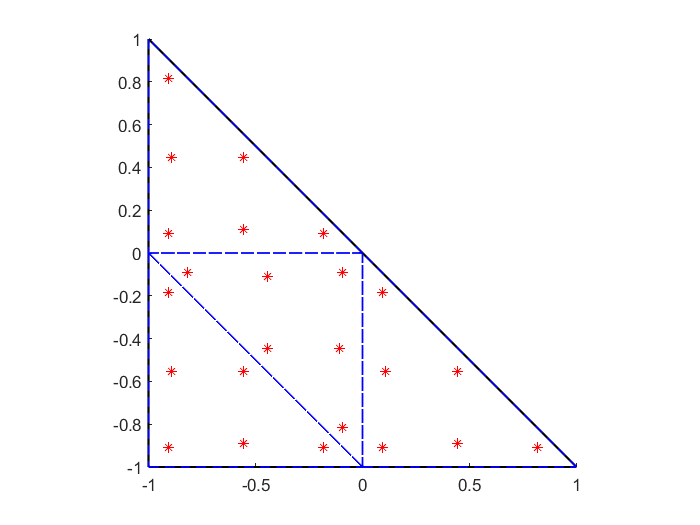}
    \end{subfigure}
    \begin{subfigure}{0.33\textwidth}
        \centering
        \includegraphics[width=1.0\textwidth]{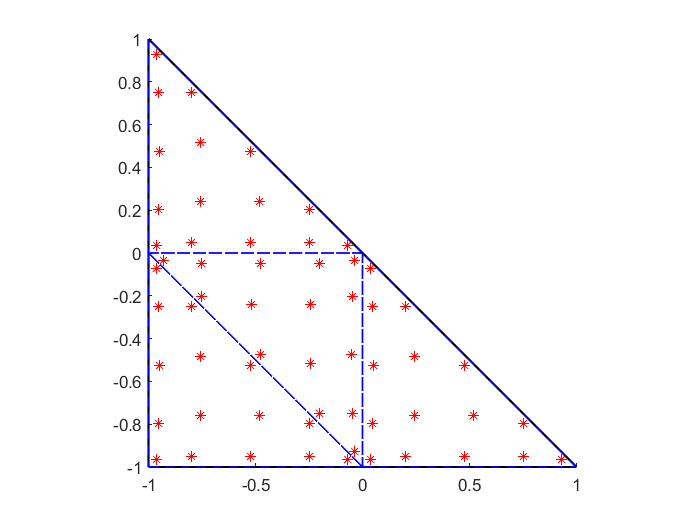}
    \end{subfigure}
        \begin{subfigure}{0.33\textwidth}
        \centering
        \includegraphics[width=1.0\textwidth]{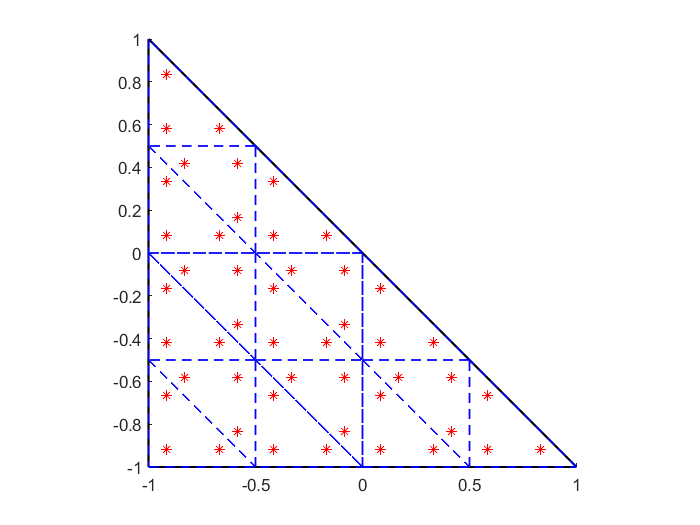}
    \end{subfigure}%
    \begin{subfigure}{0.33\textwidth}
        \centering
        \includegraphics[width=1.0\textwidth]{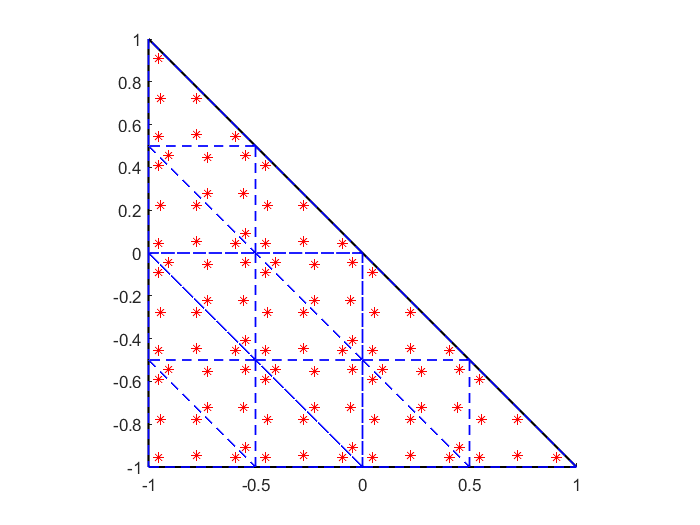}
    \end{subfigure}
    \begin{subfigure}{0.33\textwidth}
        \centering
        \includegraphics[width=1.0\textwidth]{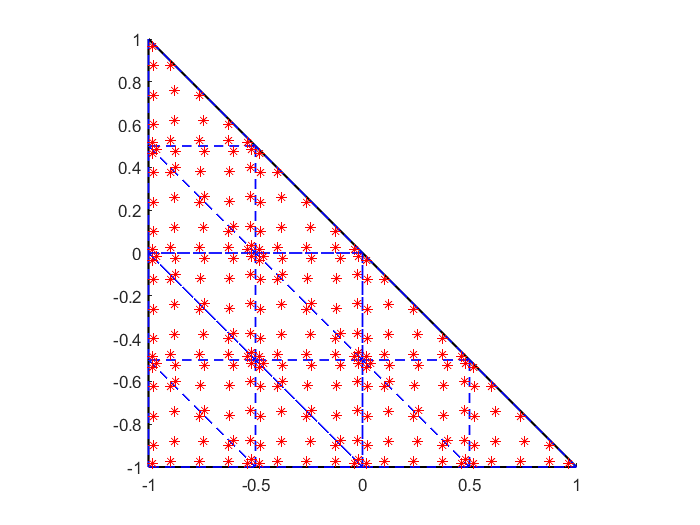}
    \end{subfigure}
    \caption{Top row, shown left to right, 3-point, 6-point, and 15-point quadrature rules on a reference triangle. Middle row, same quadrature rules as above with one $h$-refinement. Bottom row, same rules as the top row with two $h$-refinements, with respect to the top row. All quadrature rules are taken from~\cite{williams2014symmetric}.}
    \label{fig:quad_ref}
\end{figure}
In particular, `$p$-refinement' for a quadrature rule involves increasing the degree of the quadrature rule, and `$h$-refinement', involves subdividing the reference element, and applying the same strength quadrature rule on these subdivisions. In Figure \ref{fig:quad_ref}, the top row shows how `$p$-refinement' works with quadrature. Moving down each of the columns of Figure \ref{fig:quad_ref}, is what `$h$-refinement' looks like. 
\begin{figure}[h!]
\centering
    \begin{subfigure}{0.5\textwidth}
        \centering
        \includegraphics[width=1.0\textwidth]{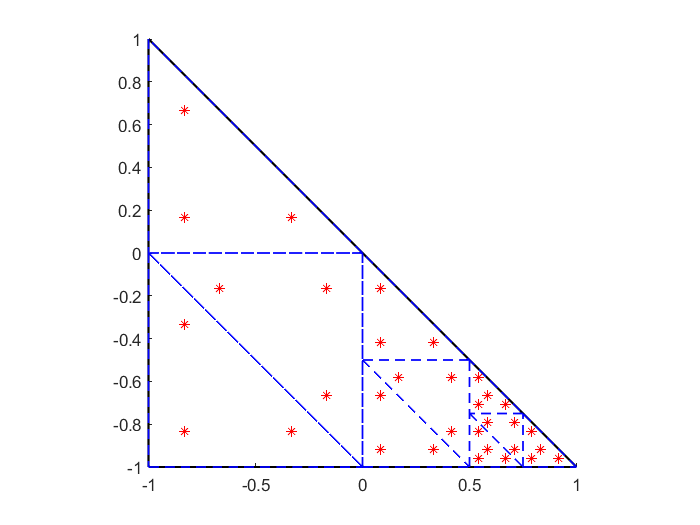}
    \end{subfigure}%
    \begin{subfigure}{0.5\textwidth}
        \centering
        \includegraphics[width=1.0\textwidth]{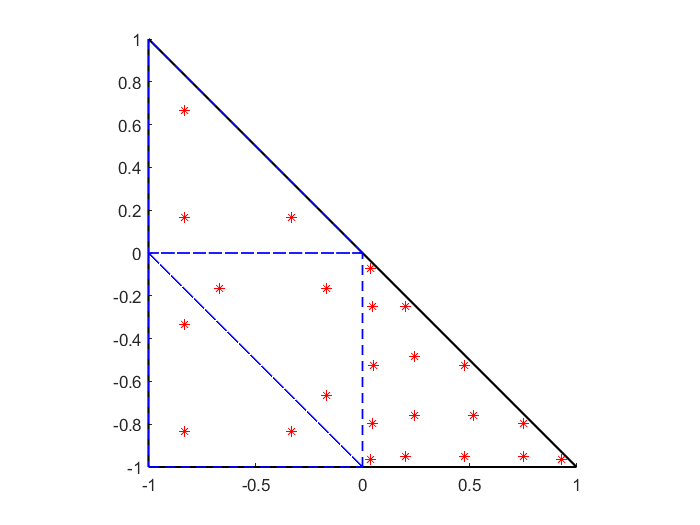}
    \end{subfigure}
    \caption{Left, a 3-point quadrature rule with non-uniform $h$-refinement. Right, a 3-point quadrature rule with one $h$-refinement, and $p$-refinement in the rightmost subtriangle.}
    \label{fig:nested_quad_ref}
\end{figure}

Based on the discussion above, it makes sense to construct an \emph{adaptive-quadrature} strategy, which leverages $h$- and $p$-refinement of quadrature rules. Naturally, this adaptive-quadrature strategy would be analogous to $hp$-adaptivity for finite element methods.
%
%
Figure~\ref{fig:nested_quad_ref} shows examples of how such a strategy could intelligently $hp$-refine to improve conservation of mass during the projection process. 
In this work, we experiment with several simplistic quadrature rules constructed manually, following this approach. However, ideally we would construct an automatic adaptive-quadrature strategy which performed $hp$-refinement based on a local error indicator. Such an error indicator could be constructed using \emph{nested-quadrature} rules. Unfortunately, to the author's knowledge, nested-quadrature rules have not yet been constructed for triangular elements. While this, at least temporarily, precludes an automatic adaptive-quadrature strategy, there is significant potential in this area, as our subsequent numerical experiments will show, (cf.~Section~\ref{mass_conservation_tests}).


\section{Solution-Transfer Theory}
\label{sec;theory}

In this section, we introduce some preliminary definitions and assumptions. Then, we present theoretical results which govern the order of accuracy of our solution-transfer method.

\subsection{Theoretical Preliminaries}

Recall that $\mathcal{T}_a$ is the source mesh and $\mathcal{T}_b$ is the target mesh. We can denote the set of all interior edges of $\mathcal{T}_a$ as $\mathcal{E} (\mathcal{T}_a)$. Next, we introduce the following space
\begin{align*}
    V^{k}_{\mathcal{T}} &= \left\{v \in L_{2}(\Omega) : v_T = v|_T \in \mathcal{P}_k (T) \quad \forall T \in \mathcal{T} \right\},
\end{align*}
where $\mathcal{T}$ is a surface mesh of triangles, and $\mathcal{P}_k$ is the space of polynomials of degree $\leq k$.

\begin{assumption}[Jump Identities]
    We assume that the following jump identities hold for all $v_h \in V^{k}_{\mathcal{T}}$, when $k\geq 1$, 
    \begin{align*}
         \left\| \llbracket v_{h} \rrbracket_e \right\|_{L_2(e)} & \lesssim C_{I} h^{k+1/2} , \\[1.5ex]
       \left\| \llbracket \partial v_{h}/\partial n \rrbracket_e \right\|_{L_2(e)} & \lesssim C_{II} h^{k-1/2}   , 
    \end{align*}
    where $C_{I} = C_{I}(v, \nabla v)$ and $C_{II} = C_{II}(v, \nabla v)$ are constants that depend on the exact solution $v$ and its first derivatives, and $e \in \mathcal{E}(\mathcal{T})$.
    \label{assumption_jump_2d}
\end{assumption}

\subsection{Theoretical Results}
\label{sec;Theoretical Results}
\noindent Let us now introduce the following space
\begin{align*}
    W_{\mathcal{T}} &= \left\{w \in H^{2}(\Omega) : w_T = w|_T \in \mathcal{HCT} (T) \quad \forall T \in \mathcal{T} \right\},
\end{align*}
where we recall that $\mathcal{HCT}$ is the space of piecewise-cubic HCT functions. Next, we can define the following smoothing operator.

\begin{definition}[$H^2$ Smoothing Operator]
    $E(\cdot): V^{k}_{\mathcal{T}} \rightarrow W_{\mathcal{T}}$ is a smoothing and synchronization operator that has the following properties
    \begin{align*}
        (E(v))(p) &= \frac{1}{|\mathcal{T}^p|}\sum_{T \in \mathcal{T}^p} v_T(p) \qquad \; \; \; \quad \forall p \in \mathcal{V} (\mathcal{T}),\\[1.5ex]
        \nabla (E(v))(p) &= \frac{1}{|\mathcal{T}^p|}\sum_{T \in \mathcal{T}^p} \nabla v_T(p) \quad \qquad \forall p \in \mathcal{V} (\mathcal{T}),\\[1.5ex]
        \frac{\partial (E(v))}{\partial n} (m) &= \frac{1}{|\mathcal{T}^m|}\sum_{T \in \mathcal{T}^m} \frac{\partial v_T}{\partial n} (m) \qquad \forall m \in \mathcal{M} (\mathcal{T}),
    \end{align*}
    where $\mathcal{T}^p = \{T \in \mathcal{T} : p \in \partial T \}$ is the set of triangles that share a vertex $p$, and $\mathcal{T}^m = \{T \in \mathcal{T} : m \in \partial T \}$ is the set of triangles that share an edge midpoint $m$.
    \label{definition_h2_2d}
\end{definition}
\begin{lemma}[Bound on $H^2$ Smoothing Operator]
    Suppose $v \in V^{k}_{\mathcal{T}}$ and $E$ is defined in accordance with Definition~\ref{definition_h2_2d}; the following result holds
    \begin{align*}
        \left\| v - E(v) \right\|^{2}_{L_2(\Omega)} & \lesssim \sum_{e \in \mathcal{E} (\mathcal{T})} \left( |e| \left\| \llbracket v \rrbracket_e \right\|^{2}_{L_2(e)} + |e|^{3} \left\| \llbracket \partial v/\partial n \rrbracket_e \right\|^{2}_{L_2(e)} \right),\\[1.5ex]
        | v - E(v) |^{2}_{H^2(\Omega ,\mathcal{T})} &\lesssim \sum_{e \in \mathcal{E} (\mathcal{T})} \left( \frac{1}{|e|^3} \left\| \llbracket v \rrbracket_e \right\|^{2}_{L_2(e)} + \frac{1}{|e|} \left\| \llbracket \partial v/\partial n \rrbracket_e \right\|^{2}_{L_2(e)} \right),
    \end{align*}
    where $e \in \mathcal{E}(\mathcal{T})$ is a generic edge.
    \label{lemma_h2_2d_k1}
\end{lemma}
\begin{proof}
    The proof appears in Corollary 2.2 of~\cite{brenner2004poincare}.
\end{proof}

Now that the smoothing operator is defined, we can introduce the following projection operator.
\begin{definition}[Projection Operator]
    $I_{h_b}(\cdot): W_{\mathcal{T}_a} \rightarrow V_{\mathcal{T}_b}^{k}$ is a projection operator that transfers the smooth solution from mesh~$\mathcal{T}_a$ to mesh $\mathcal{T}_b$.
\end{definition}

\begin{lemma}[Bound on Projection Operator] Suppose that $w \in H^{q}(\Omega)$, then
\begin{align}
    \left\| w - I_{h_b}(w) \right\|_{L_2(\Omega)} \lesssim h_{b}^{\mu} |w|_{H^{q}(\Omega)}, \label{interp_result_w}
\end{align}
where $\mu = \min(k+1, q)$. 
\label{lemma_interp_bound_2d_k1}
\end{lemma}

\begin{proof}
    This is a standard result. See for example~\cite{stogner2007c1}, section 2.3.
\end{proof}

We are now ready to introduce a theorem which governs the order of accuracy for the smoothing and projection process.

\begin{theorem}[Error Estimate]
Suppose that Assumption~\ref{assumption_jump_2d} holds, $k\geq 1$, and that the edge length $|e| \approx h_{a}$ on mesh $\mathcal{T}_a$. Then, the error between the solution on the initial mesh, $v_{h_a}$, and the smoothed/transferred solution on the final mesh, $I_{h_b}(E(v_{h_a}))$, is governed by the following inequality
\begin{align*}
    \left\| v_{h_a} - I_{h_b}(E(v_{h_a})) \right\|_{L_2(\Omega)} \lesssim \mathcal{C}_{v} h_a^{k+1}  + h^{\mu}_b \Big( \mathcal{C}_{v} h_a^{k-1} + |v_{h_a}|_{H^2(\Omega ,\mathcal{T}_a)} \Big),
\end{align*}
where $\mathcal{C}_v = \mathcal{C}_{v} (v, \nabla v)$ is a constant that depends on $v$ and its gradient, and $\mu = \min(k+1, 2)$. 
\label{error_estimate_theorem}
\end{theorem}

\begin{proof}
We begin by employing the Triangle inequality as follows
\begin{align}
    \nonumber \left\| v_{h_a} - I_{h_b}(E(v_{h_a})) \right\|_{L_2(\Omega)} &= \left\| v_{h_a} - I_{h_b}(E(v_{h_a})) + E(v_{h_a}) - E(v_{h_a}) \right\|_{L_2(\Omega)} \\[1.5ex]
    &\leq \left\| v_{h_a} - E(v_{h_a}) \right\|_{L_2(\Omega)} + \left\| E(v_{h_a}) - I_{h_b}(E(v_{h_a})) \right\|_{L_2(\Omega)}.
    \label{total_eqn}
\end{align}
There are two terms on the RHS of the expression above
\begin{align}
    &\left\| v_{h_a} - E(v_{h_a}) \right\|_{L_2(\Omega)}, \label{term_one}\\[1.5ex]
    &\left\| E(v_{h_a}) - I_{h_b}(E(v_{h_a})) \right\|_{L_2(\Omega)}. \label{term_two}
\end{align}
An upper bound for the first term (Eq.~\eqref{term_one}) can be obtained using Lemma~\ref{lemma_h2_2d_k1} 
\begin{align}
    \left\| v_{h_a} - E(v_{h_a}) \right\|_{L_2(\Omega)} \lesssim \sum_{e \in \mathcal{E} (\mathcal{T}_a)} \left( |e|^{1/2} \left\| \llbracket v_{h_a} \rrbracket_e \right\|_{L_2(e)} + |e|^{3/2} \left\| \llbracket \partial v_{h_a}/\partial n \rrbracket_e \right\|_{L_2(e)} \right).
    \label{term_one_exp}
\end{align}
An upper bound for the second term (Eq.~\eqref{term_two}) can be obtained using Lemma~\ref{lemma_interp_bound_2d_k1} with $q=2$. In particular, if we employ this lemma in conjunction with the Triangle inequality, and Lemma~\ref{lemma_h2_2d_k1} 
\begin{align}
    \nonumber &\left\| E(v_{h_a}) - I_{h_b}(E(v_{h_a})) \right\|_{L_2(\Omega)} \\[1.5ex]
    \nonumber &\lesssim h^{\mu}_b |E(v_{h_a})|_{H^2(\Omega)}\\[1.5ex]
    \nonumber &\lesssim h^{\mu}_b |E(v_{h_a})|_{H^2(\Omega, \mathcal{T}_a)}\\[1.5ex]
    \nonumber &\lesssim h^{\mu}_b \left( |E(v_{h_a}) - v_{h_a}|_{H^2(\Omega ,\mathcal{T}_a)} + |v_{h_a}|_{H^2(\Omega ,\mathcal{T}_a)} \right)\\[1.5ex]
    &\lesssim h^{\mu}_b \left( \sum_{e \in \mathcal{E} (\mathcal{T}_a)} \left( \frac{1}{|e|^{3/2}} \left\| \llbracket v_{h_a} \rrbracket_e \right\|_{L_2(e)} + \frac{1}{|e|^{1/2}} \left\| \llbracket \partial v_{h_a}/\partial n \rrbracket_e \right\|_{L_2(e)} \right) + |v_{h_a}|_{H^2(\Omega ,\mathcal{T}_a)} \right).
    \label{term_two_exp}
\end{align}
By substituting Eqs.~\eqref{term_one_exp} and~\eqref{term_two_exp} into Eq.~\eqref{total_eqn}, one obtains
\begin{align}
    \nonumber &\left\| v_{h_a} - I_{h_b}(E(v_{h_a})) \right\|_{L_2(\Omega)} \\[1.5ex] \nonumber &\lesssim \sum_{e \in \mathcal{E} (\mathcal{T}_a)} \left( |e|^{1/2} \left\| \llbracket v_{h_a} \rrbracket_e \right\|_{L_2(e)} + |e|^{3/2} \left\| \llbracket \partial v_{h_a}/\partial n \rrbracket_e \right\|_{L_2(e)} \right)\\[1.5ex] &+ h^{\mu}_b \left( \sum_{e \in \mathcal{E} (\mathcal{T}_a)} \left( \frac{1}{|e|^{3/2}} \left\| \llbracket v_{h_a} \rrbracket_e \right\|_{L_2(e)} + \frac{1}{|e|^{1/2}} \left\| \llbracket \partial v_{h_a}/\partial n \rrbracket_e \right\|_{L_2(e)} \right) + |v_{h_a}|_{H^2(\Omega ,\mathcal{T}_a)} \right).
    \label{total_eqn_exp}
\end{align}
Next, one may substitute Assumption~\ref{assumption_jump_2d} into Eq.~\eqref{total_eqn_exp}, and set $|e| \approx h_a$ as follows
\begin{align*}
    &\left\| v_{h_a} - I_{h_b}(E(v_{h_a})) \right\|_{L_2(\Omega)} \\[1.5ex] &\lesssim \sum_{e \in \mathcal{E} (\mathcal{T}_a)} h_a^{k+1} \left( C_{I} + C_{II} \right) + h^{\mu}_b \left( \sum_{e \in \mathcal{E} (\mathcal{T}_a)} h_a^{k-1} \left( C_{I} + C_{II} \right) + |v_{h_a}|_{H^2(\Omega ,\mathcal{T}_a)} \right).
\end{align*}
%
Upon observing that $h_a$ is the maximum diameter of all elements in mesh~$\mathcal{T}_a$, we obtain
\begin{align*}
    &\left\| v_{h_a} - I_{h_b}(E(v_{h_a})) \right\|_{L_2(\Omega)} \\[1.5ex] &\lesssim h_a^{k+1} \sum_{e \in \mathcal{E} (\mathcal{T}_a)}  \left( C_{I} + C_{II} \right) + h^{\mu}_b \left( h_a^{k-1} \sum_{e \in \mathcal{E} (\mathcal{T}_a)}  \left( C_{I} + C_{II} \right) + |v_{h_a}|_{H^2(\Omega ,\mathcal{T}_a)} \right).
\end{align*}
This completes the proof.

\end{proof}

\begin{remark}
    Theorem~\ref{error_estimate_theorem} will usually underpredict the order of accuracy of the smoothing/projection process. The theorem correctly predicts 2nd-order accuracy when $k =1$; however, it also predicts that the order of accuracy will be 2 when $k = 2$ or 3. In practice, we often observe that the order of accuracy is $k + 1$ for $1 \leq k \leq 3$. This is true, even though the smoothed solution is generally not guaranteed to be in $H^{k+1}(\Omega)$ when $k\geq2$. In particular, we recall that the HCT-spline space is only a subset of $H^{2}(\Omega)$, but not of $H^{q}(\Omega)$ when $q>2$. Furthermore, the exponent of $h_b$, $\mu = \min(k+1,q)$, should not equal $k+1$, when $(k+1)>q=2$. Nevertheless, order $k+1$ convergence is often observed in practice.
\end{remark}

\section{Numerical Results}
\label{sec;Results}
\noindent In this section, we attempt to answer the following practical questions:
\begin{itemize}
    \item How does synchronization and HCT interpolation during the transfer process effect the accuracy and conservation properties of the transferred solution?
    \item How does $L_2$-projection effect the accuracy and conservation properties of the transferred solution?
    \item If a discrete maximum principle is enforced via limiting of the transferred solution, how will this impact its accuracy and conservation properties?
\end{itemize}
In order to answer these questions, we investigated four different solution-transfer methods, for polynomial orders $k = 1$ and $k = 2$. 
The four methods for the $k = 1$ case are given below, alongside their abbreviations:
\begin{enumerate}
    \item TRANS1 -- No HCT smoothing, $L_2$-projection, no limiting.
    \item TRANS2 -- HCT smoothing, $L_2$-projection, no limiting, ({\bf our method}).
    \item TRANS3 -- No HCT smoothing, $L_2$-projection, with limiting.
    \item LINEAR -- Linear interpolation.
\end{enumerate}
In addition, the four methods for the $k = 2$ case are as follows: 
\begin{enumerate}
    \item TRANS1 -- No HCT smoothing, $L_2$-projection, no limiting.
    \item TRANS2 -- HCT smoothing, $L_2$-projection, no limiting, ({\bf our method}). 
    \item LINEAR -- Linear interpolation.
    \item QUADRATIC -- Quadratic interpolation.
\end{enumerate} 
Note that the limiting procedure mentioned above was taken from Alauzet and Mehrenberger,~\cite{alauzet2010p1}. It was only implemented for the $k = 1$ case because a $k = 2$ version has not yet been developed.

Our study of the performance of these solution-transfer methods is broken down into two major pieces:
\begin{itemize}
    \item A study of mass conservation and $L_2$ error.
    \item A study of visualization properties. 
\end{itemize}

\subsection{Mass Conservation and Order of Accuracy Studies}

In order to evaluate the performance of the solution-transfer methods, we tested their ability to transfer a solution between two different families of meshes. The \emph{source} family of meshes were composed from structured arrangements of triangular elements, and the \emph{target} family of meshes were composed from unstructured triangular elements. The properties of both mesh families are summarized in Table~\ref{tab:grid_info}. An example of these grids is shown in Figure \ref{fig:grid5}.
\begin{table}[h!]
\begin{center}
\begin{tabular}{|p{1.5cm}|p{2cm}|p{2cm}|p{1.5cm}|}
\hline
 & \multicolumn{3}{|c|}{Grid Size} \\
\hline
Grid Sequence Number& Structured Elements& Unstructured Elements& $h$\\
\hline
1& 32&	28&	3.535534\\
\hline
2 & 128&	124&	1.767767\\
\hline
3 & 512&	512&	0.883883\\
\hline
4 & 2048&	2064&	0.441942\\
\hline
5 & 8192&	8220&	0.220971\\
\hline
6 & 32768&	32964&	0.110485\\
\hline
7 & 131072&	130800&	0.055243\\
\hline
\end{tabular}
\end{center}
\caption{Properties of the structured and unstructured mesh families developed for the mass conservation and order of accuracy studies.}
\label{tab:grid_info}
\end{table}

In each numerical test, we started by constructing a discontinuous function on the source mesh. The purpose of this function was to mimic the numerical solution of a discontinuous finite element method. This function was produced via an $L_2$-projection of the `exact solution' ---which was itself a carefully chosen, continuous transcendental function. The resulting, discontinuous function on the source mesh is referred to as $u$. 
The transferred solution on the target mesh is referred to as $g$. 

For the mass conservation study, the mass variation (difference in mass) between the source and target grids is computed as follows
\begin{align}
    mv = \biggr|\int_{\Omega} u \hspace{1mm}d\Omega- \int_{\Omega} g \hspace{1mm}d\Omega \biggr|.
    \label{eqn:mass_variation}
\end{align}
Note that the integration is carried out element-wise using a 15-point quadrature rule capable of exactly integrating a polynomial of degree 7,~\cite{williams2014symmetric}.

For the order of accuracy study, the $L_2$ error is computed as follows
\begin{align}
     E_{L_2}= \sqrt{\int_{\Omega}\left(u-g\right)^2}.
     \label{eqn:L2_error}
\end{align}
The integral in this expression was computed over the whole domain, using the tensor product of two, 40-point Gauss quadrature rules. This yielded a sufficiently accurate approximation to the integral based on the error values at 1600 points. 

\begin{figure}[h!]
\centering
    \begin{subfigure}{0.5\textwidth}
        \centering
        \includegraphics[width=0.95\textwidth]{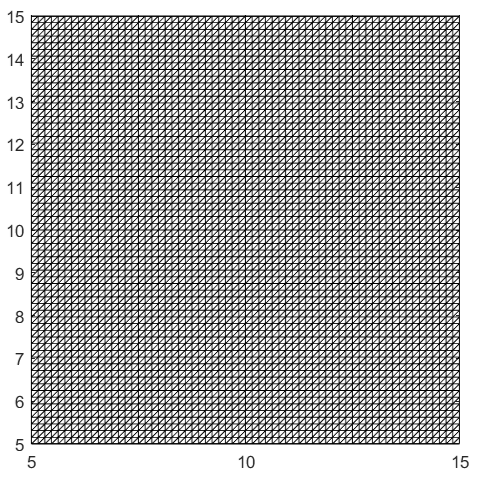}
    \end{subfigure}%
    \hfill
    \begin{subfigure}{0.5\textwidth}
        \centering
        \includegraphics[width=0.95\textwidth]{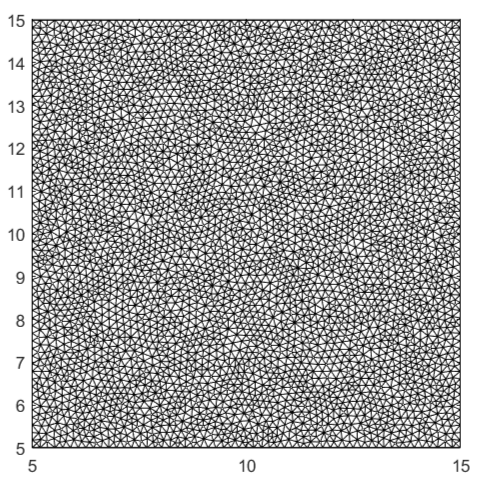}
    \end{subfigure}
    \caption{Left, structured triangular mesh with 8192 elements. Right, unstructured triangular mesh with 8220 elements. These meshes correspond to row 5 of Table~\ref{tab:grid_info}.}
    \label{fig:grid5}
\end{figure}

Three functions were selected for the study. The first, a Gaussian function
\begin{align}
    u_1 = \exp{\left(-1.5\left((x-10)^2+(y-10)^2\right)\right)}, \hspace{2mm} x,y \in [5,15]
\end{align}
was taken from Alauzet and Mehrenberger~\cite{alauzet2010p1} and was shifted to avoid any convenient cancellation from being centered about the origin. The second function
\begin{align}
    u_2 = \tanh{\left(100(y+0.3\sin{(-2x)})\right)}, \hspace{2mm} x,y \in [-1,1]
    \label{f2_def}
\end{align}
is again taken from Alauzet and Mehrenberger~\cite{alauzet2010p1} and was not shifted or changed in any way. The third function
\begin{align}
    u_3 = 12\exp{\left(-0.3\left((x-10)^2+(y-10)^2\right)\right)+\sin{(2x)}\sin{(2y)}}, \hspace{2mm} x,y \in [5,15]
\end{align}
is a return to the Gaussian function, however the Gaussian feature was made significantly larger, and smaller-scale sinusoidal wave functions were added to make the function more multi-scale in nature. The functions $u_1$, $u_2$, and $u_3$ can be seen in Figures \ref{fig:f1_L2_error}, \ref{fig:f2_L2_error}, and \ref{fig:f3_L2_error}, respectively. 

\subsubsection{Mass Conservation} \label{mass_conservation_tests}

As was mentioned in Section~\ref{sec;Implementation}, choosing the appropriate quadrature rule and refinement strategy can make a significant difference in the level of accuracy and conservation of the transferred solution. For this reason, a mass conservation study focusing on determining an appropriate quadrature-refinement strategy was carried out. This study used the first function $u_1$, polynomials of degree $k=1$, and the transfer method TRANS1 (No HCT smoothing, $L_2$-projection, no limiting). For this method, the $L_2$-projection was evaluated in conjunction with the following quadrature rules on the grids in Table~\ref{tab:grid_info}:
\begin{itemize}
    \item 15-point rule, no refinement (15 points)
    \item 15-point rule, 1 refinement (60 points)
    \item 3-point rule, no refinement (3 points)
    \item 3-point rule, 1 refinement (12 points)
    \item 3-point rule, 2 refinements (48 points)
    \item 3-point rule, 3 refinements (192 points)
    \item 6-point rule, 1 refinement (24 points)
    \item 6-point rule, 2 refinements (96 points)
\end{itemize}
Here, each `$h$-refinement' consisted of  subdividing a triangle into four subtriangles of equal size. For example, `2 refinements' involves subdividing the initial triangle into four subtriangles, and then subdividing each of these subtriangles into four subtriangles, for a total of 16 triangles.

The results of this study are shown in Figure~\ref{fig:mass_results}. The exact numerical values for this study are summarized in Table~\ref{tab:quad_study} of the Appendix.
\begin{figure}[h!]
    \centering
    \includegraphics[width=0.9\linewidth]{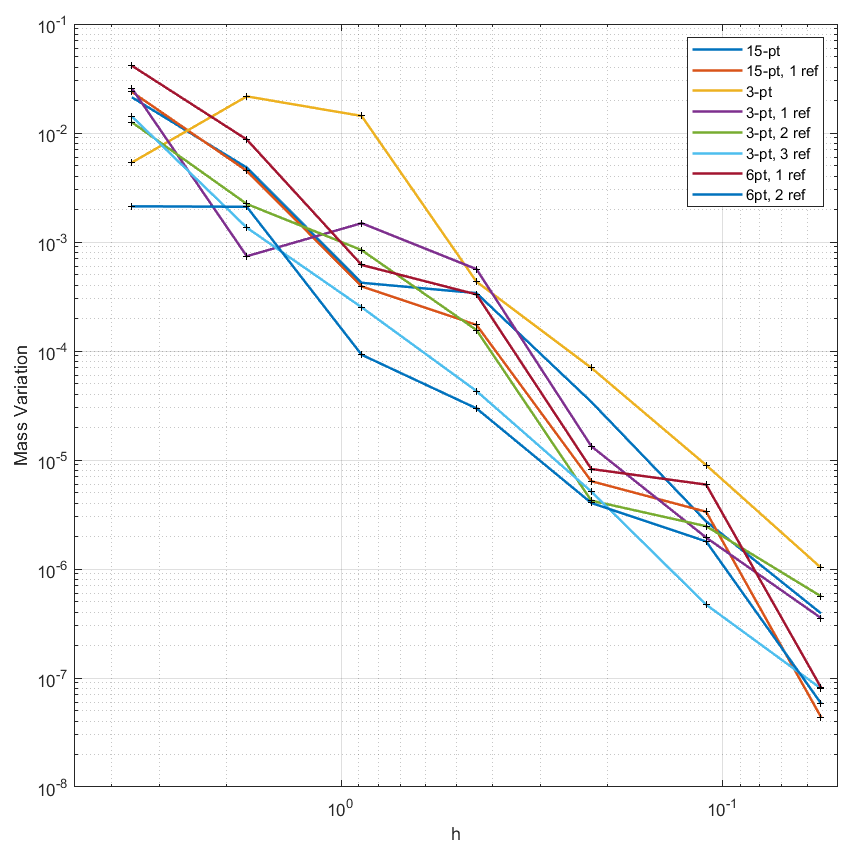}
    \caption{A study of mass conservation using the TRANS1 method on the function $u_1$. This study compares various quadrature rules to determine the most conservative rule and refinement strategy.}
    \label{fig:mass_results}
\end{figure}
The most accurate method on the finest grid was a 15-point rule with one refinement. Note that blindly making $h$-refinements is not necessarily guaranteed to provide optimal results. This is best exemplified by the 3-point rule which sees improvement in mass conservation by nearly a whole order of magnitude on the finest grid, when comparing no refinement to one refinement. Yet, a second $h$-refinement performs worse on the finest grid than one $h$-refinement, and a third $h$-refinement performs better. The effectiveness of the $h$-refinements varies (in part) due to coincidental alignment of the subdivisions with the discontinuous solution. This would suggest that using a nested quadrature rule with an error indicator to intelligently guide the refinement would yield the best results. However, because such a tool has not yet been developed (as mentioned previously), the 15-point rule with one refinement was used for the remainder of the experiments in this study. 

Now, having selected an adequate quadrature strategy, we tested the various transfer methods to determine their conservation properties.
\begin{figure}[h!]
\centering
    \begin{subfigure}{0.65\textwidth}
        \centering
        \includegraphics[width=1.0\textwidth]{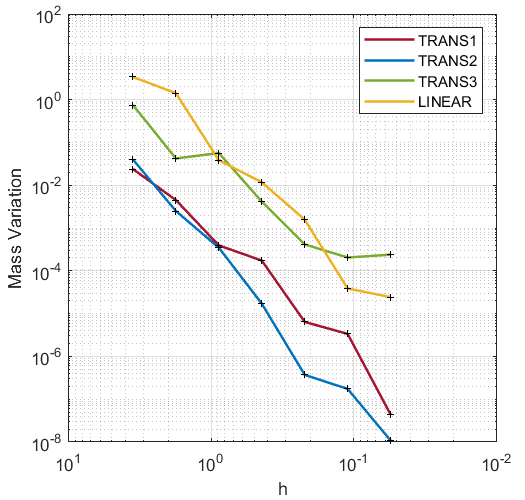}
    \end{subfigure}%
    \begin{subfigure}{0.65\textwidth}
        \centering
        \includegraphics[width=1.0\textwidth]{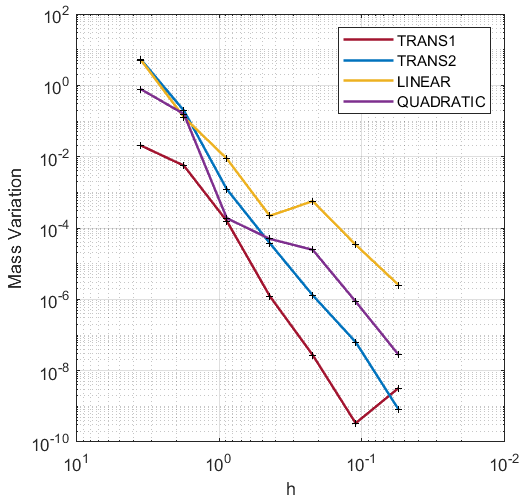}
    \end{subfigure}
    \caption{Left, mass variation of the solution transfer from structured-to-unstructured meshes for $k =1$. Right, mass variation for the $k=2$ case.}
     \label{fig:f1_mass}
\end{figure}
Figure \ref{fig:f1_mass} and Table~\ref{tab:mass_study} of the Appendix show the results for both $k = 1$ and $k = 2$ polynomial orders. The methods TRANS1 and TRANS2 perform the best overall. The $k = 2$ versions of these same transfer methods see significant improvement in conservation compared to the $k = 1$ case.
While TRANS3 is guaranteed to satisfy the discrete maximum principle, it suffers from a conservation standpoint. Similarly, LINEAR, which is the current default standard of solution transfer, did not achieve mass errors better than $10^{-6}$ in our study.

\subsubsection{Order of Accuracy}
\label{sec;Order of Accuracy}
We evaluated the $L_2$ error for all three functions ($u_1, u_2, u_3$), and all solution-transfer methods.
Figure \ref{fig:f1_L2_error} shows the results for the first function and $k =1$ polynomials.  One may observe that all methods, TRANS1, TRANS2, TRANS3, and LINEAR converge  at a rate of second order, with the TRANS1 and TRANS2 methods returning the lowest $L_2$ error on the finest grid. We note that the TRANS1  method slightly outperforms the other methods in terms of accuracy. 
\begin{figure}[h!]
\centering
    \begin{subfigure}{1.0\textwidth}
        \centering
        \includegraphics[width=1.0\textwidth]{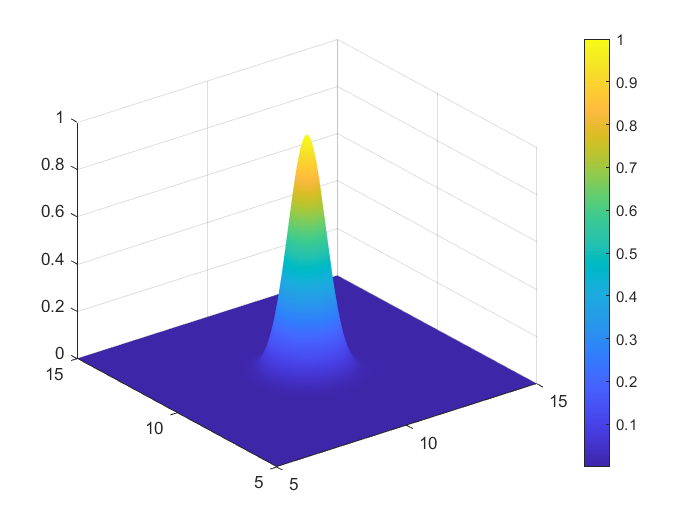}
    \end{subfigure}
    \begin{subfigure}{0.65\textwidth}
        \centering
        \includegraphics[width=1.0\textwidth]{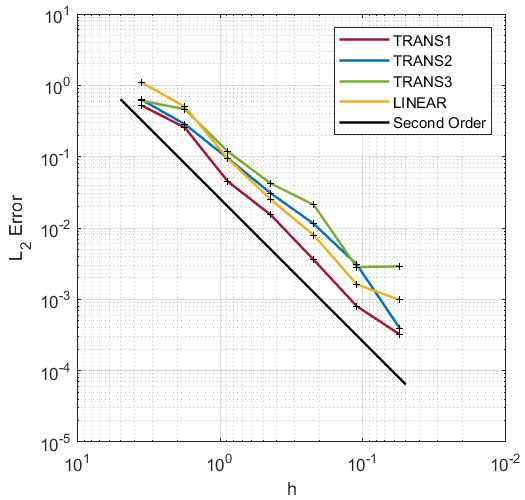}
    \end{subfigure}%
    \begin{subfigure}{0.65\textwidth}
        \centering
        \includegraphics[width=1.0\textwidth]{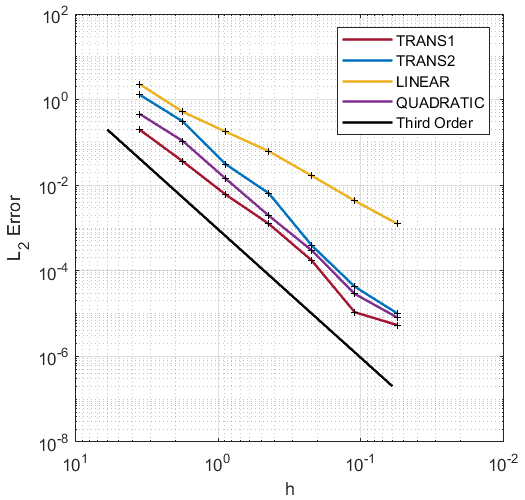}
    \end{subfigure}
    \caption{Top, the function $u_1$. Bottom left, $L_2$ error of the solution transfer from structured-to-unstructured meshes, for the $k = 1$ case. Bottom right, $L_2$ error for the $k = 2$ case.}
    \label{fig:f1_L2_error}
\end{figure}

Figure \ref{fig:f1_L2_error} also shows the results for the $k = 2$ case. Here, the TRANS1, TRANS2, and QUADRATIC methods all converge at a rate of third order. These three methods produce almost identical errors on the finest grid. However, the conventional method (the LINEAR method) still converges at a rate of second order, and produces significantly more error than the other three methods.

Next, we consider the second function. The results for this function are effected by the sharp gradients of the hyperbolic tangent in Eq.~\eqref{f2_def}. Figure~\ref{fig:f2_L2_error} shows the results for the $k = 1$ case. We observe that all methods (TRANS1, TRANS2, TRANS3, and LINEAR) converge at a first order rate until the sixth and seventh grids, where we begin to observe second order accuracy. This shift in the order of accuracy occurs because the sharp gradients of the function become fully resolved on these later grids. The results on the finest grids are consistent with those reported by Alauzet and Mehrenberger~\cite{alauzet2010p1} for the TRANS1, TRANS3, and LINEAR methods. In addition, we observe that the TRANS1 and TRANS2 methods produce similar levels of error, although TRANS1 produces slightly less.
\begin{figure}[h!]
\centering
    \begin{subfigure}{1.0\textwidth}
        \centering
        \includegraphics[width=1.0\textwidth]{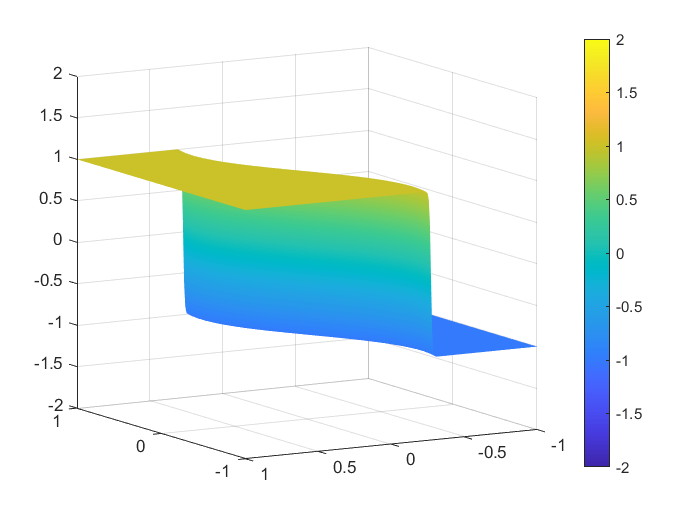}
    \end{subfigure}
    \begin{subfigure}{0.65\textwidth}
        \centering
        \includegraphics[width=1.0\textwidth]{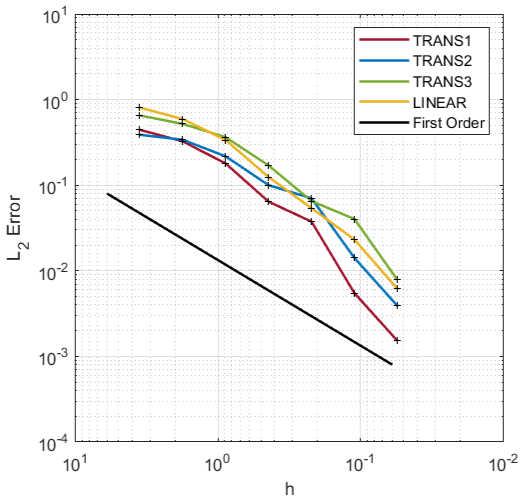}
    \end{subfigure}%
    \begin{subfigure}{0.65\textwidth}
        \centering
        \includegraphics[width=1.0\textwidth]{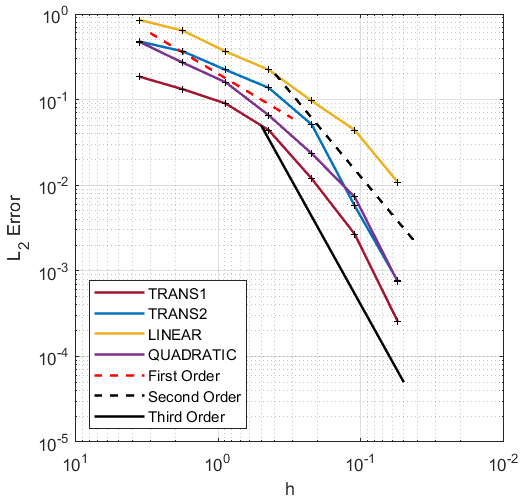}
    \end{subfigure}
    \caption{Top, the function $u_2$. Bottom left, $L_2$ error of the solution transfer from structured-to-unstructured meshes, for the $k = 1$ case. Bottom right, $L_2$ error for the $k = 2$ case.}
    \label{fig:f2_L2_error}
\end{figure}

Figure~\ref{fig:f2_L2_error} also shows the results for the $k = 2$ case. Here, the TRANS1, TRANS2, and QUADRATIC methods appear to rapidly shift from first order to second order accuracy, and eventually end at third order accuracy on the finest grids. Because of this trend, we expect to recover third order accuracy on any progressively finer grids.
Across the grids, the TRANS2 and QUADRATIC methods produce similar levels of error, and TRANS1 produces slightly less error than both methods on the finest grids.
Conversely, the LINEAR method only achieves second order accuracy on the finest grids. This is expected to be the upper limit of accuracy for the LINEAR method, regardless of the data or the grids it operates on.

Finally, we consider the third function. The results for this function are very similar to what we observed for $u_1$. Figure~\ref{fig:f3_L2_error} shows that all methods achieve second order accuracy for the $k =1$ case. In addition, for the $k = 2$ case, we observe third order accuracy for all methods except LINEAR, which once again achieves second order accuracy. 
\begin{figure}[h!]
\centering
    \begin{subfigure}{1.0\textwidth}
        \centering
        \includegraphics[width=1.0\textwidth]{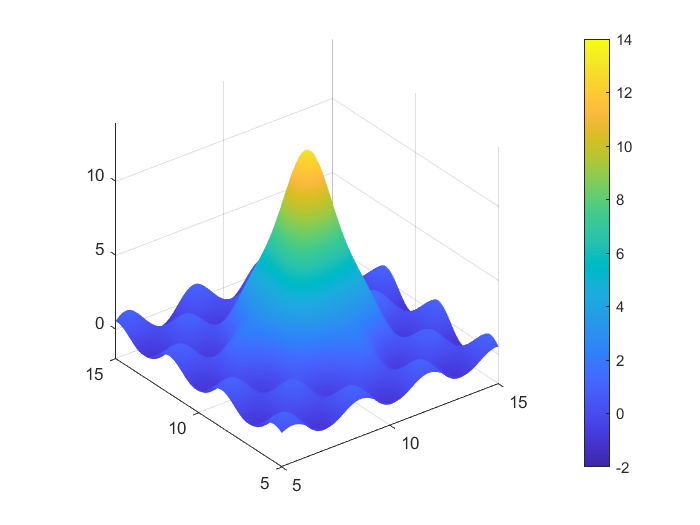}
    \end{subfigure}
    \begin{subfigure}{0.65\textwidth}
        \centering
        \includegraphics[width=1.0\textwidth]{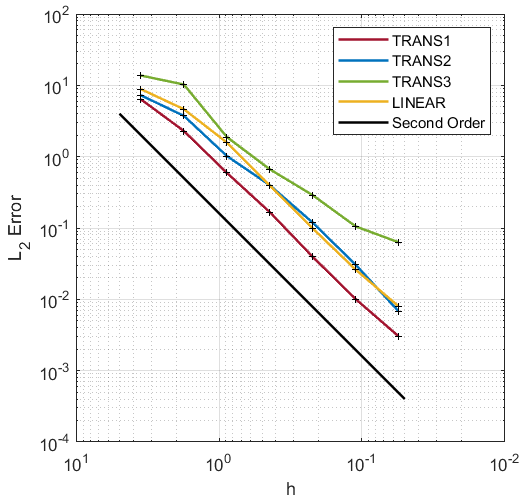}
    \end{subfigure}%
    \begin{subfigure}{0.65\textwidth}
        \centering
        \includegraphics[width=1.0\textwidth]{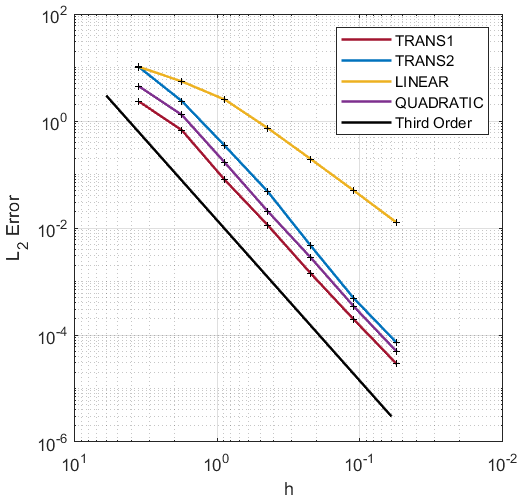}
    \end{subfigure}
    \caption{Top, the function $u_3$. Bottom left, $L_2$ error of the solution transfer from structured-to-unstructured meshes, for the $k = 1$ case. Bottom right, $L_2$ error for the $k = 2$ case.}
    \label{fig:f3_L2_error}
\end{figure}

In summary, we obtained the expected order of accuracy for each of the solution-transfer methods we tested. One may recall that Section~\ref{sec;Theoretical Results}, Theorem~\ref{error_estimate_theorem}, states that the error of the TRANS2 method is bounded by $h_a^{k+1}$ and $h_{b}^{\mu}$ where $\mu = \min(k+1,2)$. Therefore, for the $k = 1$ cases, we expected to see a convergence rate of second order, and for the $k = 2$ cases, we expected to see a minimum of second order; but possibly third order depending on the value of $\mu$. In practice, for the $k=2$ cases, we always observed third order accuracy, for sufficiently fine grids. The results of the $L_2$ error studies are reassuring, as they reinforce these theoretical results. In addition, they demonstrate that the error of the TRANS2 method is similar to that of the TRANS1 method. This indicates that the HCT-smoothing process does not have a significant impact on the accuracy of the transferred solution.

\subsection{Visualization}

The previous section demonstrates that the HCT-smoothing process---our proposed combination of synchronization and HCT interpolation---does not have a negative impact on the solution-transfer process. In particular, the mass conservation and order of accuracy studies reveal that the methods which involve smoothing or omit smoothing (i.e. TRANS2 and TRANS1, respectively) perform similarly, from a quantitative standpoint.
This begs the question: what is the \emph{qualitative} impact of the HCT-smoothing process on the solution? In what follows, we will provide evidence that the HCT-smoothing step significantly enhances the visualization properties of the solution. Towards this end, we will compare the discontinuous representations of the functions $u_2$ and $u_3$ to the HCT-smoothed versions of the same functions. Our comparisons of the smoothed and un-smoothed data are performed on an unstructured triangular grid; in particular, the fifth unstructured grid in the target mesh sequence from the previous section. The initial discontinuous data is generated via $L_2$-projection of the functions $u_2$ and $u_3$ on to the unstructured grid. For the smoothing approach, the discontinuous data is synchronized on the unstructured grid and then interpolated at the vertices and midpoints of each element using HCT cubic polynomials.

While the HCT space is $\mathcal{C}^1$-continuous, the visualization tool we used in this study, MATLAB's \textit{trisurf}, produces a piecewise linear, contour surface. In order to represent the cubic data provided by the HCT interpolation on each element, a subdivision was made at the edge midpoints to produce four subelements from each unstructured element. This allowed us to obtain a more accurate visualization of the HCT surrogate solution. 




The results for $u_2$ and the $k = 1$ case are shown in Figure \ref{fig:f2_viz_no_grid}. Here, the top row shows the exact representation of $u_2$ and its gradient magnitude, the middle row shows the discontinuous ($L_2$-projected) data and the associated gradient magnitude, and the bottom row shows the HCT-smoothed data and the associated gradient magnitude. 
The first column of Figure~\ref{fig:f2_viz_no_grid} shows how the HCT-smoothed solution (bottom left in the figure) appears to be less oscillatory, with far less extreme overshoots around the edge of the sharp feature of $u_2$, relative to the un-smoothed visualization in the row above. This difference becomes even more stark when comparing the magnitude of the gradient (second column in the figure). The un-smoothed derivative information is discontinuous, piecewise-constant data, and is very hard to visualize, whereas the HCT-smoothed data mirrors the exact representation of the gradient magnitude.  

\begin{figure}[h!]
\centering
    \begin{subfigure}{0.5\textwidth}
        \centering
        \includegraphics[width=1.0\textwidth]{F2_grid5_Exact_sol.png}
    \end{subfigure}%
    \begin{subfigure}{0.5\textwidth}
        \centering
        \includegraphics[width=1.0\textwidth]{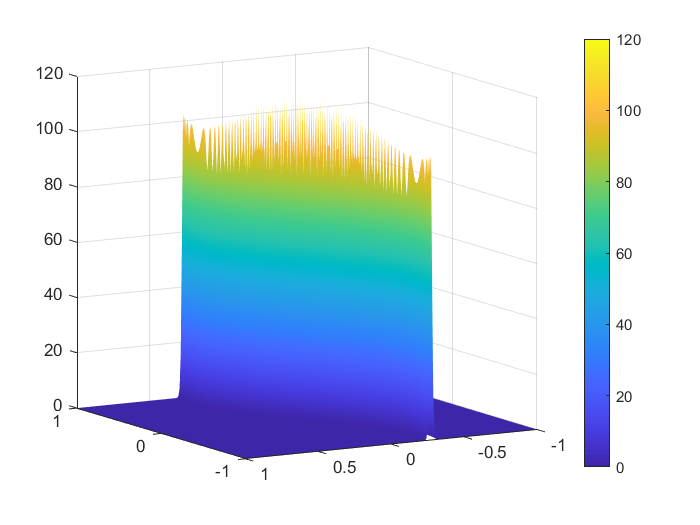}
    \end{subfigure}
    \begin{subfigure}{0.5\textwidth}
        \centering
        \includegraphics[width=1.0\textwidth]{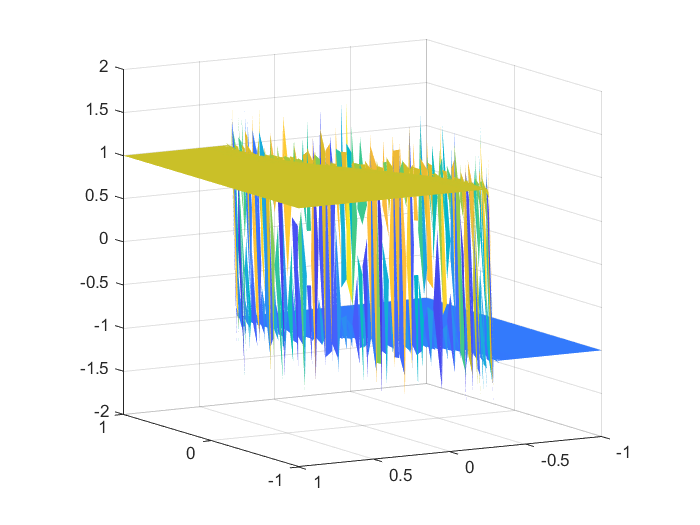}
    \end{subfigure}%
    \begin{subfigure}{0.5\textwidth}
        \centering
        \includegraphics[width=1.0\textwidth]{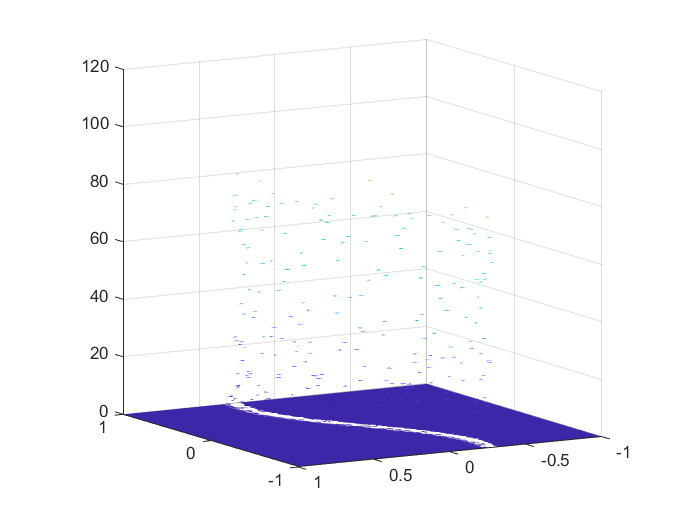}
    \end{subfigure}
    \begin{subfigure}{0.5\textwidth}
        \centering
        \includegraphics[width=1.0\textwidth]{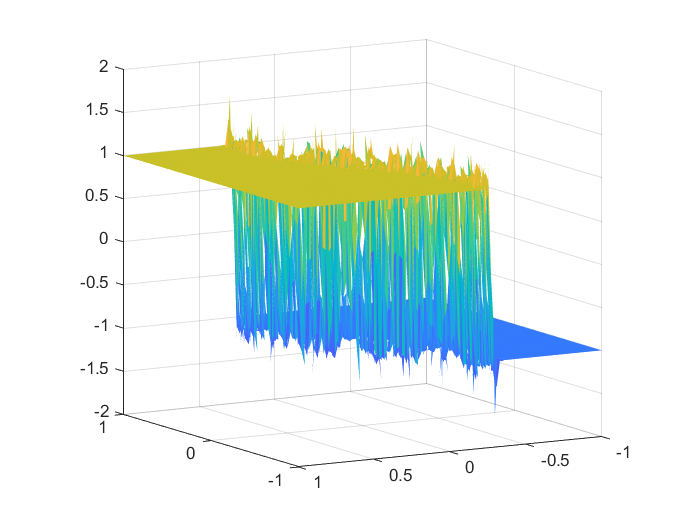}
    \end{subfigure}%
    \begin{subfigure}{0.5\textwidth}
        \centering
        \includegraphics[width=1.0\textwidth]{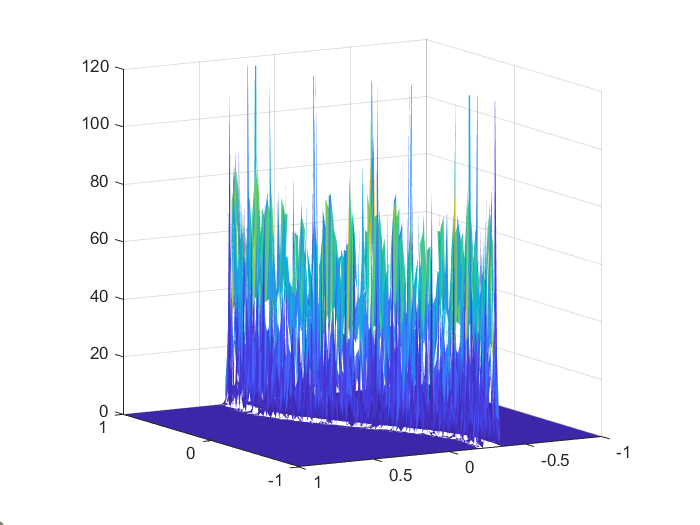}
    \end{subfigure}
    \caption{Top left, the exact representation of the function $u_2$. Top right, the magnitude of the gradient of $u_2$. Middle left, the $L_2$-projection of $u_2$. Middle right, the magnitude of the gradient of the $L_2$-projection. Bottom left, the HCT-smoothed $L_2$-projection. Bottom right, the magnitude of the gradient of the HCT-smoothed $L_2$-projection. Note: gridlines are omitted.}
    \label{fig:f2_viz_no_grid}
\end{figure}

\begin{figure}[h!]
\centering
    \begin{subfigure}{0.5\textwidth}
        \centering
        \includegraphics[width=1.0\textwidth]{F3_grid5_Exact_sol.png}
    \end{subfigure}%
    \begin{subfigure}{0.5\textwidth}
        \centering
        \includegraphics[width=1.0\textwidth]{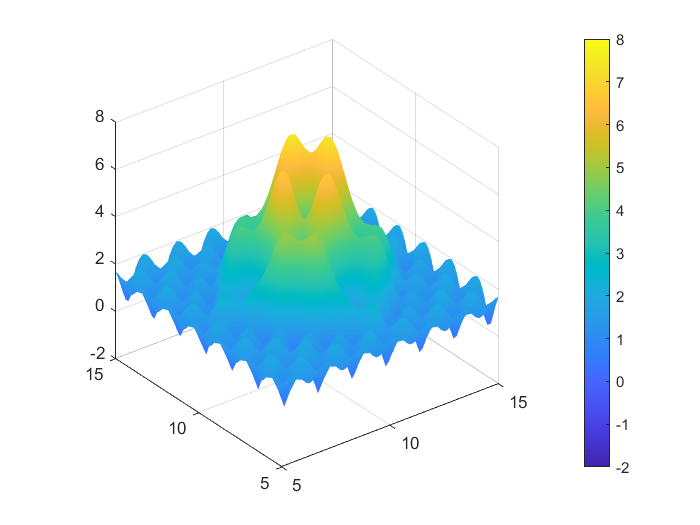}
    \end{subfigure}
    \begin{subfigure}{0.5\textwidth}
        \centering
        \includegraphics[width=1.0\textwidth]{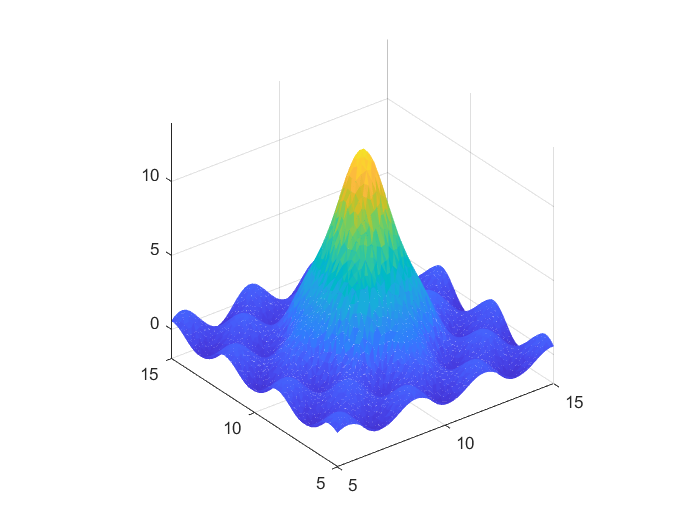}
    \end{subfigure}%
    \begin{subfigure}{0.5\textwidth}
        \centering
        \includegraphics[width=1.0\textwidth]{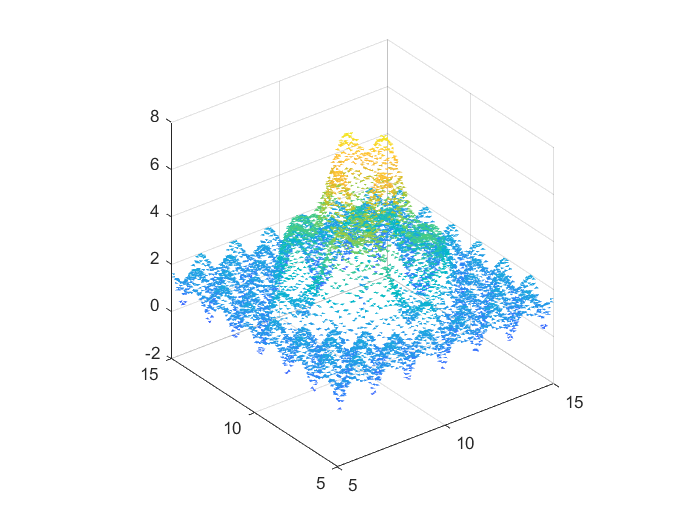}
    \end{subfigure}
    \begin{subfigure}{0.5\textwidth}
        \centering
        \includegraphics[width=1.0\textwidth]{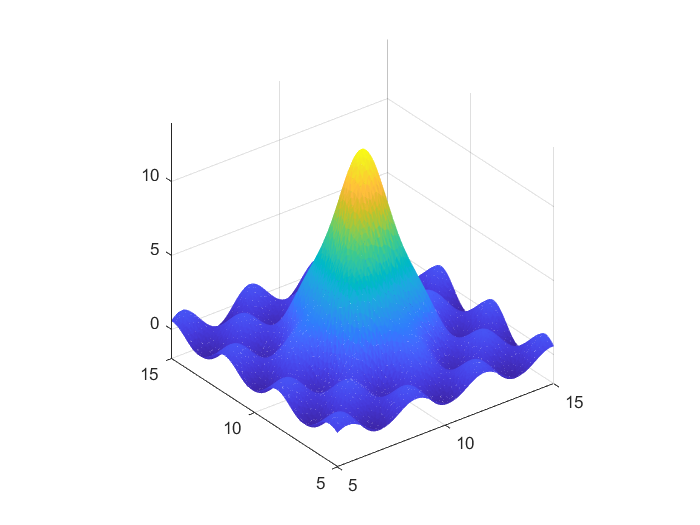}
    \end{subfigure}%
    \begin{subfigure}{0.5\textwidth}
        \centering
        \includegraphics[width=1.0\textwidth]{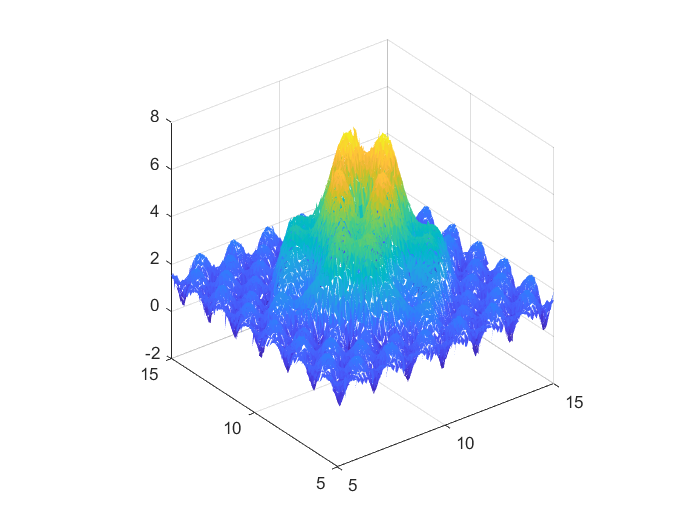}
    \end{subfigure}
    \caption{Top left, the exact representation of the function $u_3$. Top right, the magnitude of the gradient of $u_3$. Middle left, the $L_2$-projection of $u_3$. Middle right, the magnitude of the gradient of the $L_2$-projection. Bottom left, the HCT-smoothed $L_2$-projection. Bottom right, the magnitude of the gradient of the HCT-smoothed $L_2$-projection. Note: gridlines are omitted.}
    \label{fig:f3_viz_no_grid}
\end{figure}

The results for $u_3$ and the $k = 1$ case are shown in Figure~\ref{fig:f3_viz_no_grid}. In this case, the HCT-smoothed results are still better than the un-smoothed results, but the difference is less pronounced. This is due to the fact that $u_3$ is a smoother function than $u_2$. Nevertheless, the visualization of the gradient magnitude is still noticeably superior for the HCT-smoothed results. 

We also evaluated the HCT-smoothing method for different values of $k$ and grid resolutions, for the function $u_2$.
The top of Figure~\ref{fig:HCT_u2_ref} shows how $k = 2$ polynomials on the fifth unstructured grid can improve the quality of the solution and the gradient magnitude for the HCT-smoothed results. Even more impressive results are obtained in the presence of mesh refinement. This is shown in the bottom row of Figure~\ref{fig:HCT_u2_ref} for $k =1$ polynomials, and the seventh unstructured grid. 


\begin{figure}[h!]
\centering
    \begin{subfigure}{0.5\textwidth}
        \centering
        \includegraphics[width=1.0\textwidth]{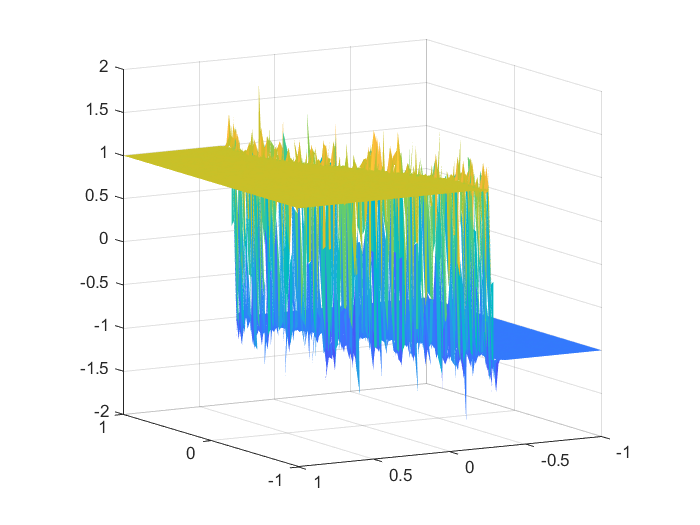}
    \end{subfigure}%
    \begin{subfigure}{0.5\textwidth}
        \centering
        \includegraphics[width=1.0\textwidth]{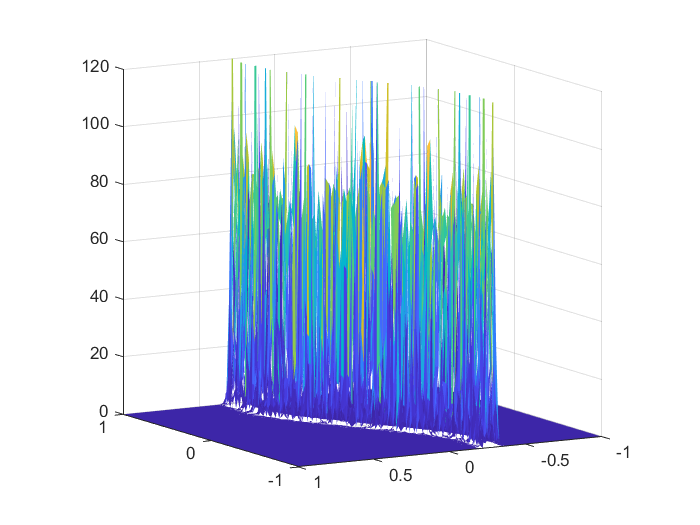}
    \end{subfigure}
    \begin{subfigure}{0.5\textwidth}
        \centering
        \includegraphics[width=1.0\textwidth]{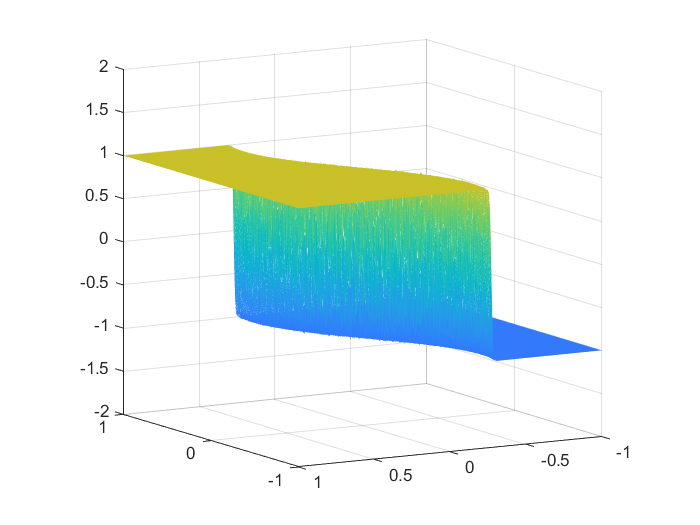}
    \end{subfigure}%
    \begin{subfigure}{0.5\textwidth}
        \centering
        \includegraphics[width=1.0\textwidth]{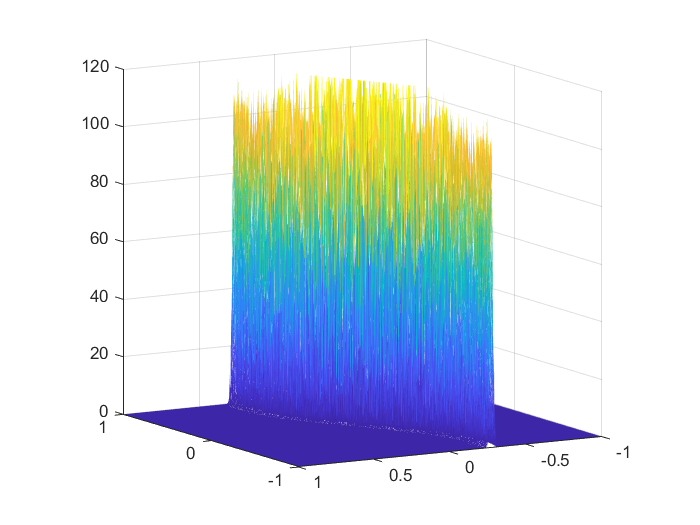}
    \end{subfigure}
    \caption{Top left, the HCT-smoothed $L_2$-projection of $u_2$ on unstructured grid 5 with $k = 2$. Top right, the magnitude of the gradient for the HCT-smoothed $L_2$-projection. Bottom left, the HCT-smoothed $L_2$-projection of $u_2$ on unstructured grid 7 with $k=1$. Bottom right, the magnitude of the gradient for the HCT-smoothed $L_2$-projection. Note: gridlines are omitted.}
    \label{fig:HCT_u2_ref}  
\end{figure}

\pagebreak
\clearpage

\section{Conclusion}
\label{sec;Conclusion}
In this work, we have presented a solution-transfer method for slab-based space-time finite element methods which uses HCT splines. This solution-transfer method was motivated by recent developments in space-time mesh adaptation which have made solution transfer between time-slabs non-trivial. The HCT solution-transfer method satisfies $\mathcal{P}_k$ exactness for $k \leq 3$; it is capable of maintaining the discrete maximum principle; it is capable of enforcing conservation (asymptotically); and it provides a smooth, continuous surrogate solution. These features make the proposed method a reliable method for many solution-transfer applications. However, it is particularly well-suited (by design) for space-time problems because it facilitates the enforcement of space-time boundary conditions \emph{and} more convenient solution visualization. 
In what follows, we briefly review the key contributions of the current paper. 

We derived an error bound for the smoothing and interpolation procedure (steps 1 and 2 of the transfer process) which predicts an order of accuracy of 2 when $1 \leq k \leq 3$.

Numerical experiments were conducted to test the conservation, accuracy, and visualization properties of the HCT solution-transfer method. 
The results of the conservation experiments showed that for $k=1$ and $k=2$, HCT solution-transfer is conservative (to within approximately $10^{-8}$ to $10^{-10}$) on fine grids. In addition, experiments suggest that with a sufficiently strong quadrature rule, HCT can be conservative on coarse grids as well. In addition, quadrature rules with various strengths and refinement levels were isolated and tested in this paper to illustrate how the conservation properties of the transfer method are dependent on the properties of the quadrature rules. 

The order of accuracy experiments showed that the HCT transfer method converged at second order accuracy for $k=1$ and third order accuracy for $k=2$, for all three transcendental functions.
These results showed that the theoretical error bound accurately predicts the order of accuracy when $k=1$, while it may  underpredict the order of accuracy for larger values of $k$.

The visualization study qualitatively showed the improved ability of the smooth HCT surrogate solution to visualize smooth derivatives compared to the discrete visualization provided by $L_2$-projection. Additionally, for solutions with sharp gradients, HCT smoothing/interpolation outperforms $L_2$-projection by providing a solution with smaller oscillations around sharp features.

Future work includes extending our spline-based transfer approach to higher dimensions, and developing an adaptive-quadrature routine which is fully conservative, to within machine precision. 

\pagebreak
\clearpage

\section*{CRediT authorship contribution statement}

\textbf{L. Larose}: Writing – original draft, Writing – review and editing, Visualization, Methodology, Formal
analysis. \textbf{J.T. Anderson}: Conceptualization, Formal
analysis. \textbf{D.M. Williams}: Writing – original draft, Writing – review and editing, Conceptualization, Formal
analysis, Supervision, Funding acquisition.

\section*{Software Availability}

The solution-transfer algorithms were implemented as an extension of the JENRE$^{\text{\textregistered}}$  Multiphysics Framework~\cite{Cor19_SCITECH}, which is United States government-owned software developed by the Naval Research Laboratory with other collaborating institutions. This software is not available for public use or dissemination.

\section*{Declaration of Competing Interests}

The authors declare that they have no known competing financial interests or personal relationships that could have appeared to influence the work reported in this paper.

\section*{Funding}

This research received funding from the United States Naval Research Laboratory (NRL) under grant number N00173-22-2-C008. In turn, the NRL grant itself was funded by Steven Martens, Program Officer for the Power, Propulsion and Thermal Management Program, Code 35, in the United States Office of Naval Research.


\appendix

\section{Mass Conservation Results}
\label{sec;Appendix}

\begin{landscape}
\centering
\begin{table}
    \begin{tabular}{|p{1.75cm}|p{1.75cm}|p{1.75cm}|p{1.75cm}|p{1.75cm}|p{1.75cm}|p{1.75cm}|p{1.75cm}|p{1cm}|}
    \hline
    \multicolumn{8}{|c|}{Conservation Error}&
    \multicolumn{1}{|c|}{Grid Size} \\
    \hline
    3 Point Rule No Refinement& 3 Point Rule 1 Refinement & 3 Point Rule 2 Refinements & 3 Point Rule 3 Refinements & 6 Point Rule 1 Refinement& 6 Point Rule 2 Refinements & 15 Point Rule No Refinement& 15 Point Rule 1 Refinement & $h$\\
    \hline
    5.3645E-03& 2.5605E-02& 1.2526E-02& 1.4160E-02&	4.1365E-02&	2.1254E-03&2.1237E-02& 2.3790E-02& 3.5355\\
    \hline
    2.1678E-02&	7.4051E-04&	2.2322E-03&	1.3544E-03&	8.7070E-03&	2.1038E-03& 4.8127E-03&	4.4905E-03&	1.7678\\
    \hline
    1.4324E-02&	1.4883E-03&	8.4237E-04&	2.5164E-04&	6.1587E-04&	9.2341E-05& 4.2259E-04&	3.9149E-04& 0.8839\\
    \hline
    4.3249E-04&	5.6351E-04&	1.5551E-04&	4.2937E-05&	3.3023E-04&	2.9746E-05& 3.3970E-04&	1.7353E-04&	0.4419\\
    \hline
    7.0019E-05&	1.3369E-05&	4.2376E-06&	5.1364E-06&	8.2661E-06&	4.0359E-06& 3.4119E-05&	6.3893E-06&	0.2210\\
    \hline
    8.9699E-06&	1.9417E-06&	2.4684E-06&	4.7072E-07&	5.9463E-06&	1.7891E-06& 2.7322E-06&	3.3483E-06&	0.1105\\
    \hline
    1.0265E-06&	3.5504E-07&	5.6141E-07&	8.0003E-08& 8.1899E-08&	5.8543E-08& 3.9180E-07&	4.3819E-08&	0.0552\\
    \hline
    \end{tabular}
    \caption{Mass conservation errors for different quadrature and refinement combinations, for the function $u_1$, polynomial order $k=1$, and the transfer method TRANS1.}
    \label{tab:quad_study}
\end{table}
\vspace{3cm}

\begin{table}
    \begin{tabular}{|p{1.8cm}|p{1.8cm}|p{1.8cm}|p{1.8cm}|p{1.8cm}|p{1.8cm}|p{1.8cm}|p{2.25cm}|p{1.75cm}|}
    \hline
    \multicolumn{8}{|c|}{Mass Variation}&
    \multicolumn{1}{|c|}{Grid Size} \\
    \hline
    TRANS1 $k = 1$& TRANS2 $k = 1$& TRANS3 $k = 1$& LINEAR $k = 1$& TRANS1 $k = 2$& TRANS2 $k = 2$& LINEAR $k = 2$ & QUADRATIC& $h$\\
    \hline
    2.3790E-02	&3.9848E-02	&7.4641E-01	&3.3734E+00	&2.0470E-02	&5.3024E+00	&5.1074E+00	&7.7767E-01	&3.5355\\
    \hline
    4.4905E-03	&2.4907E-03	&4.2112E-02	&1.4128E+00	&5.5934E-03	&1.9573E-01	&1.2811E-01	&1.5673E-01	&1.7678\\
    \hline
    3.9149E-04	&3.4686E-04	&5.6514E-02	&3.8382E-02	&1.5639E-04	&1.2164E-03	&8.7809E-03	&1.8765E-04	&0.8839\\
    \hline
    1.7353E-04	&1.7059E-05	&4.1335E-03	&1.1711E-02	&1.2299E-06	&3.7658E-05	&2.1934E-04	&4.9561E-05	&0.4419\\
    \hline
    6.3893E-06	&3.7086E-07	&4.1256E-04	&1.5655E-03	&2.7372E-08	&1.2963E-06	&5.6613E-04	&2.4421E-05	&0.2210\\
    \hline
    3.3483E-06	&1.7384E-07	&2.0474E-04	&3.8443E-05	&3.3725E-10	&6.3448E-08	&3.3924E-05	&8.6629E-07	&0.1105\\
    \hline
    4.3819E-08	&1.0804E-08	&2.3813E-04	&2.4234E-05	&3.2649E-09	&8.1159E-10	&2.5202E-06	&2.7770E-08	&0.0552\\
    \hline
    \end{tabular}
    \caption{Mass conservation errors for different solution-transfer methods, for the function $u_1$, and the polynomial orders $k=1$ and $k=2$.}
    \label{tab:mass_study}
\end{table}
\end{landscape}

\pagebreak
\clearpage
{\footnotesize\bibliography{technical-refs}}

\end{document}